\documentclass[11pt]{amsart} 
\usepackage{amscd}
\usepackage{amsmath,empheq}
\usepackage{amsfonts}
\usepackage{amssymb}
\usepackage{mathrsfs}
\usepackage[all]{xy}
\newtheorem{thm}{Theorem}[section]

\newtheorem{prop}[thm]{Proposition}

\theoremstyle{definition}

\theoremstyle{remark}
\newtheorem{rem}[thm]{Remark}
\newtheorem{example}[thm]{Example}
\def\divv{\textrm{div}}
\def\RR{\mathbb{R}}
\def\NN{\mathbb{N}}
\def\qed{\hfill $\Box$\par\vskip3mm}
\newenvironment{proof-sketch}{\noindent{\bf Sketch of Proof}\hspace*{1em}}{\qed\bigskip}

\everymath{\displaystyle}
\begin{document}
\title{Anisotropic equations with indefinite
 potential \\ and competing nonlinearities }
\author[N.S. Papageorgiou]{Nikolaos S. Papageorgiou}
\address[N.S. Papageorgiou]{Department of Mathematics,
National Technical University, 
				Zografou Campus, 15780 Athens, Greece \& Institute of Mathematics, Physics and Mechanics, 1000 Ljubljana, Slovenia}
\email{\tt npapg@math.ntua.gr}
\author[V.D. R\u{a}dulescu]{Vicen\c{t}iu D. R\u{a}dulescu}
\address[V.D. R\u{a}dulescu]{Faculty of Applied Mathematics, AGH University of Science and Technology, al. Mickiewicza 30, 30-059 Krak\'ow, Poland \& Department of Mathematics, University of Craiova, 200585 Craiova, Romania
\& Institute of Mathematics, Physics and Mechanics, 1000 Ljubljana, Slovenia}
\email{\tt radulescu@inf.ucv.ro}
\author[D.D. Repov\v{s}]{Du\v{s}an D. Repov\v{s}}
\address[D.D. Repov\v{s}]{Faculty of Education and Faculty of Mathematics and Physics, University of Ljubljana
\& Institute of Mathematics, Physics and Mechanics, 1000 Ljubljana, Slovenia}
\email{\tt dusan.repovs@guest.arnes.si}
\begin{abstract}
We consider a nonlinear Dirichlet problem driven by a variable exponent $p$-Laplacian plus an indefinite potential term. The reaction has the competing effects of a parametric concave (sublinear) term and of a convex (superlinear) perturbation (an anisotropic concave-convex problem). We prove a bifurcation-type theorem describing the changes in the set of positive solutions  as the positive parameter $\lambda$ varies. Also, we prove the existence of minimal positive solutions.
\end{abstract}
\maketitle

 \section{Introduction}

Let $\Omega\subseteq\RR^N$ be a bounded domain with a $C^2$-boundary $\partial\Omega$. In this paper we study the following anisotropic boundary value problem
\begin{equation}\tag{\mbox{$P_\lambda$}}
\left\{
\begin{array}{lll}
-\Delta_{p(z)}u(z)+\xi(z)u(z)^{p(z)-1}=\lambda u(z)^{q(z)-1}+f(z,u(z)) \text{ in } \Omega,\\
u|_{\partial\Omega}=0,\;\lambda>0,u>0.
\end{array}
\right.
\end{equation}

In this problem, $\Delta_{p(z)}$ denotes the $p(z)$-Laplacian defined by
$$
\Delta_{p(z)}u={\rm div}\,(|Du|^{p(z)-2}Du(z)) \mbox{ for all }u\in W^{1,p(z)}_0(\Omega).
$$

Concerning the exponents $p,q:\overline{\Omega}\to\RR$, we assume that both are functions belonging to $C^1(\overline{\Omega})$ and we have
$$
1<q_-\leq q(z)\leq q_+<p_-\leq p(z)\leq p_+ \mbox{ for all }z\in\overline{\Omega}.
$$

The potential function $\xi\in L^\infty(\Omega)$ is sign-changing. So, the differential operator of $(P_\lambda)$ (left-hand side) is not coercive. In the reaction (right-hand side of $(P_\lambda)$), we have a parametric term with $\lambda>0$ being the parameter and a perturbation $f(z,x)$ which is jointly measurable and of class $C^1$ in the $x$-variable. We assume that $f(z,\cdot)$ exhibits $(p_+-1)$-superlinear growth near $+\infty$ without satisfying the usual in such cases Ambrosetti-Rabinowitz condition ($AR$-condition for short). So, in the reaction of problem $(P_\lambda)$ we have the competing effects of a sublinear (concave) term and of a superlinear (convex) term. We are looking for positive solutions and our aim is to have a precise description of the changes in the set of positive solutions as the parameter $\lambda>0$ varies (a bifurcation-type result).

The study of such parametric concave-convex problems started with the seminal paper of Ambrosetti, Brezis and Cerami \cite{1Amb-Bre-Cer}, where $p(z)=2$ for all $z\in\overline{\Omega}$ (semilinear isotropic problem). It was extended to equations driven by the $p$-Laplacian and with the reaction being $\lambda x^{q-1}+x^{r-1}$ for all $x\geq0$ with $1<q<p<r<p^*$
 by Garcia Azorero, Manfredi and Peral Alonso \cite{8Gar-Man-Per}, and Guo and Zhang \cite{11Guo-Zha}.
 Recall that $$p^*=\left\{
                                  \begin{array}{ll}
                                    \frac{Np}{N-p}, & \hbox{ if } p\leq N \\
                                    +\infty, & \hbox{ if } N<p.
                                  \end{array}
                                \right.$$
 Further extensions can be found in the works of Marano and Papageorgiou \cite{14Mar-Pap} and Papageorgiou and R\u adulescu \cite{15Pap-Rad}. All the aforementioned works deal with  isotropic equations. To the best of our knowledge, no such results exist for anisotropic equations.

Additional parametric boundary value problems driven by operators with variable exponents and applications, can be found in the book of R\u adulescu and Repov\v s \cite{18Rad-Rep}. We also refer to the recent papers \cite{BRR1,BRR2,Cherfils,PRR0,PRR1,PRR2,PRR2bis,PRR3,ZRJMPA}, all dealing with isotropic or anisotropic nonlinear problems with Dirichlet boundary condition.

\section{Mathematical background, auxiliary results and hypotheses}

In this section we briefly review some basic facts about variable exponent spaces and we prove two anisotropic strong comparison theorems which we will need in our analysis of problem $(P_\lambda)$.

A comprehensive presentation of variable exponent Lebesgue and Sobolev spaces can be found in the book of Diening, Harjulehto, H\"asto and Ruzi\v cka \cite{4Die-Har-Has-Ruz}.

So, let $L_1^\infty(\Omega)=\{p\in L^\infty(\Omega):\: {\rm essinf}_\Omega\, p\geq 1\}$.
For $p\in L_1^\infty (\Omega)$, we set
$$p_-={\rm essinf}_\Omega\, p\quad\mbox{and}\quad p_+={\rm esssup}_\Omega\, p.$$

Also let $M(\Omega)=\{u:\Omega\to\RR:\: u(\cdot)\ \mbox{is measurable}\}$.
As usual, we identify two such functions which differ on a set of zero measure.

Given $p\in L_1^\infty(\Omega)$, we define the following variable exponent Lebesgue space
$$
L^{p(z)}(\Omega)=\left\{u\in M(\Omega):\: \int_\Omega |u|^{p(z)}dz<+\infty\right\}.
$$

We equip $L^{p(z)}(\Omega)$ with the following norm (known as the Luxemburg norm)
$$
\|u\|_{p(z)}=\inf\left\{\lambda>0:\:\int_\Omega \left(\frac{|u|}{\lambda}\right)^{p(z)}dz\leq1\right\}.
$$

Having defined variable exponent Lebesgue spaces, we can introduce variable exponent Sobolev spaces by
$$
W^{1,p(z)}(\Omega)=\{u\in L^{p(z)}(\Omega):\: |Du|\in L^{p(z)}(\Omega)\}.
$$

We equip this space with the following norm
$$
\|u\|_{1,p(z)}=\|u\|_{p(z)}+\|Du\|_{p(z)}.
$$

An equivalent norm of $W^{1,p(z)}(\Omega)$ is given by
$$
\|u\|'_{1,p(z)}=\inf\left\{\lambda>0:\:\int_\Omega \left(\left(\frac{|Du|}{\lambda}\right)^{p(z)}+\left(\frac{|u|}{\lambda}\right)^{p(z)}\right)dz\leq1
\right\}.
$$

We define $W^{1,p(z)}_0(\Omega)$ as the closure in the $\|\cdot\|_{1,p(z)}$ of all compactly supported $W^{1,p(z)}(\Omega)$-functions.

When $p\in L_1^\infty(\Omega)$ and $p_->1$, then the spaces $L^{p(z)}(\Omega)$, $W^{1,p(z)}(\Omega)$ and $W^{1,p(z)}_0(\Omega)$ are all separable, reflexive and uniformly convex.

We set
$$
p^*(z)=\left\{
         \begin{array}{ll}
           \frac{Np(z)}{N-p(z)}, & \hbox{ if } p(z)< N \\
           +\infty, & \hbox{ if }p(z)\geq N.
         \end{array}
       \right.
$$

If $p,q\in C(\overline{\Omega})$, $p_+<N$ and $1\leq q(z)\leq p^*(z)$ (resp. $1\leq q(z)<p^*(z)$) for all $z\in\overline{\Omega}$, then $W^{1,p(z)}(\Omega)$ and $W^{1,p(z)}_0(\Omega)$ are embedded continuously (resp. compactly) into $L^{q(z)}(\Omega)$.

If $p,p'\in L_1^\infty(\Omega)$ and $\frac{1}{p(z)}+\frac{1}{p'(z)}=1$, then $L^{p(z)}(\Omega)^*=L^{p'(z)}(\Omega)$ and we have the following H\"older type inequality
$$
\int_\Omega |uv|dz\leq\left(\frac{1}{p_-}+\frac{1}{p'_-}\right)\|u\|_{p(z)}\|v\|_{p'(z)} \mbox{ for all }u\in L^{p(z)}(\Omega),\:v\in L^{p'(z)}(\Omega).
$$

We say that  $p\in C(\overline{\Omega})$ is logarithmic H\"older continuous (denoted by $p\in C^{0,\frac{1}{|\ln t|}}$) if it satisfies
$$
\left|p(z)-p(z')\right|\leq \frac{c}{|\ln|z-z'||} \mbox{ for some } c>0, \mbox{ all }z,z'\in\Omega, |z-z'|\leq\frac{1}{2}.
$$

Note that $C^{0,1}(\overline{\Omega})\hookrightarrow C^{0,\frac{1}{|\ln t|}}(\overline{\Omega})$. Also, when $p\in C^{0,\frac{1}{|\ln t|}}(\overline{\Omega})$, then
$$
W^{1,p(z)}_0(\Omega)=\overline{C_c^\infty(\overline{\Omega})}^{\|\cdot\|_{1,p(z)}}.
$$

Moreover, in this case the Poincar\'e inequality holds and we have
$$
\|u\|_{p(z)}\leq\hat{C}\|Du\|_{p(z)} \mbox{ for all }u\in W^{1,p(z)}_0(\Omega),
$$
where $\hat{C}>0$ depends only on $(p,N,|\Omega|_N,{\rm diam}\,\Omega)$, with $|\cdot|_N$ denoting the Lebesgue measure on $\RR^N$. So, when $p\in C^{0,\frac{1}{|\ln t|}}(\overline{\Omega})$, then on the Sobolev space $W^{1,p(z)}_0(\Omega)$, we can use the equivalent norm
$$
\|u\|=\|Du\|_{p(z)} \mbox{ for all }u\in W^{1,p(z)}_0(\Omega).
$$

We introduce the following modular functions
\begin{eqnarray*} 
  && \rho(u)=\int_\Omega |u|^{p(z)}dz \mbox{ for all }u\in L^{p(z)}(\Omega), \\
  && \hat{\rho}(Du)=\int_\Omega |Du|^{p(z)}dz \mbox{ for all }u\in W^{1,p(z)}_0(\Omega).
\end{eqnarray*}

We have the following property.
\begin{prop}\label{prop1}
 (a) For $u\in L^{p(z)}(\Omega)$, $u\not=0$, we have
$$
\|u\|_{p(z)}=\lambda \Leftrightarrow \rho\left(\frac{u}{\lambda}\right)=1;
$$

   (b) $\|u\|_{p(z)}<1$  (resp. $=1$, $>1$)  $\Leftrightarrow$ $\rho(u)<1$ (resp. $=1$, $>1$);

    (c) $\|u\|_{p(z)}<1$ $\Rightarrow$ $\|u\|_{p(z)}^{p_+}\leq \rho(u)\leq\|u\|_{p(z)}^{p_-}$  and
               $\|u\|_{p(z)}>1$ $\Rightarrow$ $\|u\|_{p(z)}^{p_-}\leq \rho(u)\leq\|u\|_{p(z)}^{p_+}$;

    (d) $\|u_n\|_{p(z)}\to0$ $\Leftrightarrow$ $\rho(u_n)\to0$;

    (e) $\|u_n\|_{p(z)}\to+\infty$ $\Leftrightarrow$ $\rho(u_n)\to+\infty$.
\end{prop}

Similarly, we have the following implications, when $p\in C^{0,\frac{1}{|\ln t|}}(\overline{\Omega})$.

\begin{prop}\label{prop2}
 (a) For $u\in W^{1,p(z)}_0(\Omega)$, $u\not=0$, we have
$$
\|u\|=\lambda \Leftrightarrow \hat{\rho}\left(\frac{Du}{\lambda}\right)=1;
$$

    (b) $\|u\|<1$ (resp. $=1$, $>1$) $\Leftrightarrow$ $\hat{\rho}(Du)<1$ (resp. $=1$, $>1$);

    (c) $\|u\|<1$ $\Rightarrow$ $\|u\|^{p_+}\leq \hat{\rho}(Du)\leq \|u\|^{p_-}$ and
                $\|u\|>1$ $\Rightarrow$ $\|u\|^{p_-}\leq \hat{\rho}(Du)\leq\|u\|^{p_+}$;

    (d) $\|u_n\|\to 0$ $\Leftrightarrow$ $\hat{\rho}(Du_n)\to0$;
    (e) $\|u_n\|\to+\infty$ $\Leftrightarrow$ $\hat{\rho}(Du_n)\to+\infty$.
\end{prop}

Let $p\in C^{0,\frac{1}{|\ln t|}}(\overline{\Omega})$. Then
$$
W^{1,p(z)}_0(\Omega)^*=W^{-1,p'(z)}(\Omega) \quad \left(\frac{1}{p(z)}+\frac{1}{p'(z)}=1\right).
$$

Consider the operator $A:W^{1,p(z)}_0(\Omega)\to W^{-1,p'(z)}(\Omega)=W^{1,p(z)}_0(\Omega)^*$ defined by
$$
\langle A(u),h\rangle=\int_\Omega |Du|^{p(z)-2}(Du, Dh)_{\RR^N}dz \mbox{ for all }u,h\in W^{1,p(z)}_0(\Omega).
$$

This operator has the following properties (see Gasinski and Papageorgiou \cite{9Gas-Pap}).

\begin{prop}\label{prop3}
  The map $A:W^{1,p(z)}_0(\Omega)\to W^{-1,p'(z)}(\Omega)$ defined above is bounded (that is, maps bounded sets to bounded sets), continuous, strictly monotone (hence maximal monotone, too) and of type $(S)_+$, that is $u_n\overset{w}{\to}u$ in $W^{1,p(z)}_0(\Omega)$ and $\displaystyle{\limsup_{n\to\infty}\langle A(u_n),(u_n-u)\rangle \leq0}$ $\Rightarrow$ $u_n\to u$ in $W^{1,p(z)}_0(\Omega)$.
\end{prop}

Next, we prove two strong comparison theorems, which will be used in the analysis of problem $(P_\lambda)$. The first one  extends Proposition 2.6 of Arcoya and Ruiz \cite{3Arc-Rui} to the abstract setting of anisotropic problems .

We will use  the following notation. Given $h,g \in L^\infty(\Omega)$, we write that $h\prec g$ if and only if for every $K\subseteq \Omega$ compact, we can find $c_K>0$ such that $0<c_K\leq g(z)-h(z)$ for a.a. $z\in K$. Evidently, if $h,g\in C(\Omega)$ and $h(z)<g(z)$ for all $z\in \Omega$, then $h\prec g$. Also, by $C_+$ we denote the positive cone of $C_0^1(\overline{\Omega})=\{u\in C^1(\overline{\Omega}):\: u|_{\partial\Omega}=0\}$, that is, $C_+=\{u\in C_0^1(\overline{\Omega}):\:u(z)\geq0 \mbox{ for all }z\in\overline{\Omega}\}$. This cone has a nonempty interior given by
$$
{\rm int}\,C_+=\left\{u\in C_+:\: u(z)>0 \mbox{ for all }z\in\Omega,\ \frac{\partial u}{\partial n}|_{\partial\Omega}<0\right\},
$$
with $n(\cdot)$ being the outward unit normal on $\partial\Omega$.
\begin{prop}\label{prop4}
  If $p\in C^1(\overline{\Omega})$, $1<p(z)$ for all $z\in\overline{\Omega}$, $\hat{\xi},h,g\in L^\infty(\Omega)$, $\hat{\xi}(z)\geq0$ for a.a. $z\in\Omega$, $h\prec g$, $u\in W^{1,p(z)}(\Omega)$, $u\not=0$, $v\in {\rm int}\,C_+$ and
\begin{eqnarray*} 
  && -\Delta_{p(z)}u+\hat{\xi}(z)|u|^{p(z)-2}u=h(z) \mbox{ in }\Omega,\:u|_{\partial\Omega}\leq0, \\
  && -\Delta_{p(z)}v +\hat{\xi}(z)v^{p(z)-1}=g(z) \mbox{ in } \Omega,\:\frac{\partial v}{\partial n}|_{\partial\Omega}<0,
\end{eqnarray*}
then $v-u\in{\rm int}\,C_+$.
\end{prop}

\begin{proof}
From Theorem 4.1 of Fan and Zhao \cite{6Fan-Zha}  (see also Proposition 3.1 of Gasinski and Papageorgiou \cite{9Gas-Pap}), we have that $u\in L^\infty(\Omega)$. Then invoking Theorem 1.3 of Fan \cite{5Fan}, we infer that $u\in C^1(\overline{\Omega})$.
Also exploiting the monotonicity of $A(\cdot)$ (see Proposition \ref{prop3}), we see that $u\leq v$.
We introduce the following two sets.
$$
E=\{z\in\Omega:\:u(z)=v(z)\} \mbox{ and } \hat{E}=\{z\in\Omega:\: Du(z)=Dv(z)=0\}.
$$

{\it Claim}. $E\subseteq \hat{E}$.

Let $y=u-v$. We have $y\leq0$. Consider $z\in E$. Then $y(z)=\displaystyle{\max_{\Omega}y=0}$. So, we have $Dy(z)=0$, hence $Du(z)=Dv(z)$. Arguing by contradiction, suppose that $z\not\in \hat{E}$. Then  $Dv(z)\not=0$ and so we can find an open ball $B\subseteq \Omega$ centered at $z$ such that
$$
|Du(x)|>0,\;|Dv(x)|>0,\;(Du(x),Dv(x))_{\RR^N}>0 \mbox{ for all }x\in B.
$$

Consider the $N\times N$ matrix $A(x)=\left(a_{ij}(x)\right)_{i,j=1}^{N}$ with entries $a_{ij}(x)$ defined by
$$
a_{ij}(x)=\int_0^1\left[(1-t)Du(x)+tDv(x)\right]\left[\delta_{ij}+(p(x)-2)
\frac{D_i((1-t)u+tv)D_j((1-t)u+tv)}{|(1-t)Du+tDv|^2}\right]dt
$$

We have $a_{ij}\in C^{0,\alpha}(\overline{B})$ for some $\alpha\in(0,1)$ (see Fan \cite{5Fan}) and
\begin{equation}\label{eq1}
  -\divv \left(A(x)Dy(x)\right)=h(x)-g(x)-\hat{\xi}(x)\left[|u(x)|^{p(x)-2}u(x)-v(x)^{p(x)-1}\right] \mbox{ in } B
\end{equation}
(see also Guedda and V\'eron \cite{10Gue-Ver}). By choosing the ball $B\subseteq\Omega$ even smaller if necessary, we obtain that in \eqref{eq1} the linear differential operator is strictly elliptic, while the right-hand side is strictly negative. Invoking Theorem 4 of V\'azquez \cite{20Vaz}, we have
\begin{eqnarray*} 
   && y(x)<0 \mbox{ for all }x\in B, \\
   &\Rightarrow& y(z)<0, \mbox{ a contradiction.}
\end{eqnarray*}

Therefore $z\in\hat{E}$ and this proves the claim.

Recall that $v\in{\rm int}\,C_+$. Therefore $\hat{E}$ is compact and so we can  find $U\subseteq\Omega$ such that
\begin{equation}\label{eq2}
  \hat{E}\subseteq U\subseteq \overline{U}\subseteq\Omega.
\end{equation}

For $\varepsilon>0$ small, we have
\begin{eqnarray} 
  && u(z)+\varepsilon<v(z) \mbox{ for all }z\in\partial U \mbox{ (see \eqref{eq2}),} \label{eq3} \\
  && h(z)+\varepsilon<g(z) \mbox{ for a.a. }z\in U \mbox{ (recall that $h\prec g$). } \label{eq4}
\end{eqnarray}

We choose $\delta>0$ small so that
\begin{eqnarray}\nonumber 
   && \left|\hat{\xi}(z)\left(|x|^{p(z)-2}x-|w|^{p(z)-2}w\right)\right| \\
   &\leq& \|\hat{\xi}\|_\infty \left||x|^{p(z)-2}x-|w|^{p(z)-2}w\right|<\varepsilon \label{eq5}\\ \nonumber
       && \mbox{ if } |x-w|<\delta,\;z\in \overline{U} \mbox{ (recall that $p\in C^1(\overline{\Omega})$).}
\end{eqnarray}

Then we have for a.a. $z\in U$
\begin{eqnarray*} 
   && -\Delta_{p(z)}(u+\delta)+\hat{\xi}(z)|u+\delta|^{p(z)-2}(u+\delta) \\
  &=& -\Delta_{p(z)} u+\hat{\xi}(z)|u+\delta|^{p(z)-2}(u+\delta) \\
  &=& h(z)+\hat{\xi}(z)\left[|u+\delta|^{p(z)-2}(u+\delta)-|u|^{p(z)-2}u\right] \\
  &\leq& h(z)+\varepsilon \mbox{ (see \eqref{eq5}) } \\
  &<& g(z) \mbox{ (see \eqref{eq4}) } \\
  &=& -\Delta_{p(z)} v +\hat{\xi}(z)v^{p(z)-1}, \\
  &\Rightarrow& u(z)+\delta\leq v(z) \mbox{ for all }z\in\overline{U} \\
  && \mbox{ (by the weak comparison principle, see \eqref{eq3}), } \\
  &\Rightarrow& u(z)<v(z) \mbox{ for all }z\in \overline{U}, \\
  &\Rightarrow& E=\emptyset.
\end{eqnarray*}

Also, from the anisotropic maximum principle of Zhang \cite{21Zha}, we have
\begin{eqnarray*} 
  && \frac{\partial y}{\partial n}|_{\partial\Omega}>0, \\
  &\Rightarrow& \frac{\partial(v-u)}{\partial n}|_{\partial\Omega}<0.
\end{eqnarray*}
The proof is now complete.
\end{proof}

For the second strong comparison principle, we use the following open cone in $C^1(\overline{\Omega})$
$$
D_+=\{u\in C^1(\overline{\Omega}):\: u(z)>0 \mbox{ for all }z\in\Omega, u|_{\partial\Omega\cap U^{-1}(0)}<0\}.
$$

\begin{prop}\label{prop5}
  If $p\in C^1(\overline{\Omega})$, $1<p(z)$ for all $z\in\overline{\Omega}$, $\hat{\xi},h,g\in L^\infty(\Omega)$, $\hat{\xi}(z)\geq0$ for a.a. $z\in\Omega$, $0<\eta\leq g(z)-h(z)$ for a.a. $z\in\Omega$ and $u,v\in C^1(\overline{\Omega})$ satisfy $u\leq v$ and
\begin{eqnarray*} 
   && -\Delta_{p(z)}u+\hat{\xi}(z)|u|^{p(z)-2}u=h(z)\mbox{ in }\Omega, \\
   && -\Delta_{p(z)}v +\hat{\xi}(z)|v|^{p(z)-2}v=g(z) \mbox{ in } \Omega,
\end{eqnarray*}
then $v-u\in D_+$.
\end{prop}
\begin{proof}
  The reasoning is similar to that of the previous proposition.

Let $w=v-u\geq0$, $w\in C^1(\overline{\Omega})$. As in the proof of Proposition \ref{prop4}, we have
\begin{equation}\label{eq6}
  -\divv(A(z)Dw)=g(z)-h(z)-\hat{\xi}(z)\left[|v|^{p(z)-2}v-|u|^{p(z)-2}u\right] \mbox{ in }\Omega.
\end{equation}

In this case we have $a_{ij}\in W^{1,\infty}(\Omega)$ for all $i,j=1,...,N$.

Suppose that for some $z_0\in\Omega$, we have $w(z_0)=0$. Then $u(z_0)=v(z_0)$. The map $(z,x)\mapsto|x|^{p(z)-2}x$
is uniformly continuous on $\overline{\Omega}\times\RR$. So, we can find $\delta>0$ small such that
$$
g(z)-h(z)-\hat{\xi}(z)\left||v(z)|^{p(z)-2}v(z)-|u(z)|^{p(z)-2}u(z)\right|\geq\frac{\eta}{2}>0
$$
for a.a. $z\in B_\delta(z_0)=\{z\in\Omega:\:|z-z_0|<\delta\}$.

From \eqref{eq6} we have
$$
-\divv \left(A(z)Dw\right)\geq\frac{\eta}{2}>0 \mbox{ for a.a. }z\in B_\delta(z_0).
$$

Invoking Theorem 4 of V\'azquez \cite{20Vaz}, we have
$$
w(z)>0 \mbox{ for all }z\in B_\delta(z_0),
$$
a contradiction to the fact that $w(z_0)=0$. Therefore
$$
w(z)>0 \mbox{ for all }z\in\Omega.
$$

Let $E_0=\{z\in\partial\Omega:\:w(z)=0\}$. We can assume that $E_0\not=\emptyset$. Otherwise we already have that
 $v(z)>u(z)$ for all $z\in\overline{\Omega}$ and we are done. By Zhang \cite{21Zha} we have
\begin{eqnarray*} 
  && \frac{\partial w}{\partial n}(z_0)<0, \\
  &\Rightarrow& w=v-u\in D_+.
\end{eqnarray*}
The proof is now complete.
\end{proof}

Now we introduce our hypotheses on the data of problem $(P_\lambda)$.

\smallskip
$H_0$: $p,q\in C^1(\overline{\Omega})$ and $1<q_-\leq q(z)\leq q_+<p_-\leq p(z)\leq p_+$ for all $z\in\overline{\Omega}$,
$\xi\in L^\infty(\Omega)$.

\smallskip
$H_1$: $f:\Omega\times\RR\to\RR$ is a function which for all $x\in\RR$, is measurable in $z\in\Omega$, for a.a. $z\in\Omega$ we have $f(z,\cdot)\in C^1(\RR)$ and
\begin{itemize}
  \item[(i)] $a\leq f(z,x)\leq a(z)(1+x^{r(z)-1})$ for a.a. $z\in\Omega$, all $x\in\RR$, with $a\in L^\infty(\Omega)$, $r\in C(\overline{\Omega})$ and $p_+<r(z)<p(z)^*$ for all $z\in\overline{\Omega}$;
  \item[(ii)] if $F(z,x)=\displaystyle{\int_0^x f(z,s)ds}$, then $\displaystyle{\lim_{x\to+\infty}\frac{F(z,x)}{x^{p_+}}=+\infty}$ uniformly for a.a. $z\in\Omega$;
  \item[(iii)] if $e(z,x)=f(z,x)x-p_+F(z,x)$, then there exist $M>0$ and $\hat{C}>\|\xi\|_\infty\left(\frac{p_+}{p_-}-1\right)$ such that
$$
e'_x(z,x)\geq\hat{C}x^{p(z)-1} \mbox{ for a.a. }z\in\Omega, \mbox{ all }x\geq M;
$$
  \item[(iv)] $\displaystyle{\lim_{x\to0^+}\frac{f(z,x)}{x^{p_- -1}}=0}$ uniformly for a.a. $z\in\Omega$.
\end{itemize}

\begin{rem}\label{rem1}
Since we are looking for positive solutions and all the above hypotheses concern the positive semi-axis $\RR_+=[0,+\infty)$, without any loss of generality, we may assume that
\begin{equation}\label{eq7}
  f(z,x)=0 \mbox{ for a.a. }z\in\Omega, \mbox{ all }x\leq0.
\end{equation}
\end{rem}

\begin{rem}\label{rem2}
Hypothesis $H_1(ii)$ implies that for a.a. $z\in\Omega$, the function $F(z,\cdot)$ is $p_+$-superlinear. This combined with  hypothesis $H_1(iii)$, says that for a.a. $z\in\Omega$, $f(z,\cdot)$ is $(p_+-1)$-superlinear. However, the superlinearity of $f(z,\cdot)$ is not expressed using the usual for such problems $AR$-condition (see Ambrosetti and Rabinowitz \cite{2Amb-Rab}). Instead we employ the less restrictive condition $H_1(iii)$ that permits the consideration of $(p_+-1)$-superlinear functions with ``slower" growth near $+\infty$, which fail to satisfy the $AR$-condition (see the examples below). Note that hypothesis $H_1(iii)$ is not global and implies that for a.a. $z\in\Omega$, eventually $e(z,\cdot)$ is nondecreasing. Then Lemma 2.4(iv) of Li and Yang \cite{13Li-Yan}, implies that there exists $\mu\in L^1(\Omega)$ such that
\begin{equation}\label{eq8}
  e(z,x)\leq e(z,y)+\mu(z) \mbox{ for a.a. }z\in\Omega, \mbox{ all }0\leq x\leq y
\end{equation}
(a global quasi-monotonicity condition on $e(z,\cdot)$). Also it is equivalent to saying that there exists $\hat{M}>0$ such that for a.a. $z\in\Omega$, the function
$$
x\mapsto\frac{f(z,x)}{x^{p_+-1}}
$$
is nondecreasing on $[\hat{M},+\infty)$ (see Li and Yang \cite{13Li-Yan}).
\end{rem}

\begin{example}\label{examples}
  Consider the following two functions
$$
f_1(z,x)=x^{r(z)-1} \mbox{ for all }x\geq0,
$$
$$
f_2(z,x)=\left\{
           \begin{array}{ll}
             \hat{C}x^{r(z)-1}, & \hbox{ if }x\in[0,1] \\
             \hat{C}(x^{p(z)-1}+x^{m(z)-1}), & \hbox{ if }1<x
           \end{array}
         \right. \mbox{ (see \eqref{eq7}). }
$$
with $r,m\in C(\overline{\Omega})$, $p_+<r(z)<p(z)^*$, $m(z)\leq p(z)$ and $r(z)=p(z)+m(z)-2$.

Note that $f_1(z,\cdot)$ satisfies the $AR$-condition, while $f_2(z,\cdot)$ need not (the $AR$-condition is not satisfied if $\{z\in\overline{\Omega}:\:m(z)=p_+\}$ has nonempty interior).
\end{example}

Finally, we mention that if $X$ is a Banach space and $\varphi\in C^1(X,\RR)$, then $K_\varphi$ denotes the critical set of $\varphi$, that is,
$$
K_\varphi=\{u\in X:\:\varphi'(u)=0\}.
$$

Moreover, a set $S\subseteq W^{1,p(z)}_0(\Omega)$ is said to be ``downward directed", if given $u,v\in S$, we can find $w\in S$ such that $w\leq u$, $w\leq v$. In addition if $u,v\in W^{1,p(z)}_0(\Omega)$ with $u\leq v$, then we define
\begin{eqnarray*} 
  && [u,v]=\{y\in W^{1,p(z)}_0(\Omega):\: u(z)\leq y(z)\leq v(z) \mbox{ for a.a. }z\in\Omega\} \\
  && {\rm int}_{C_0^1(\overline{\Omega})}[u,v]=\mbox{ the interior in }C_0^1(\overline{\Omega}) \mbox{ of }[u,v]\cap C_0^1(\overline{\Omega}), \\
  && [u)=\{y\in W^{1,p(z)}_0(\Omega):\:u(z)\leq y(z) \mbox{ for a.a. }z\in\Omega\}.
\end{eqnarray*}

\section{Positive solutions}

We introduce the following two sets:
\begin{eqnarray*} 
   && \mathcal{L}=\{\lambda>0:\mbox{ problem $(P_\lambda)$ admits a positive solution}\}, \\
   && S_\lambda=\mbox{ set of positive solutions of problem $(P_\lambda)$.}
\end{eqnarray*}

Let $\vartheta>\|\xi\|_\infty$ (see hypothesis $H_0$), $\lambda>0$ and consider the functional $\hat{\varphi}_\lambda: W^{1,p(z)}_0(\Omega)\to\RR$ defined by
\begin{eqnarray*} 
  \hat{\varphi}_\lambda(u) &=& \int_\Omega \frac{1}{p(z)}|Du|^{p(z)}dz+\int_\Omega \frac{\xi(z)}{p(z)}|u|^{p(z)}dz+\int_\Omega \frac{\vartheta}{p(z)}(u^-)^{p(z)}dz \\
  &-& \lambda\int_\Omega \frac{1}{q(z)}(u^+)^{q(z)}dz-\int_\Omega F(z,u^+)dz \mbox{ for all }u\in W^{1,p(z)}_0(\Omega).
\end{eqnarray*}

Recall that $u^+=\max\{u,0\}$, $u^-=\max\{-u,0\}$ and if $u\in W^{1,p(z)}_0(\Omega)$, then
$$
u^+,u^-\in W^{1,p(z)}_0(\Omega),\;u=u^+-u^-,\;|u|=u^+ +u^-.
$$

We have that $\hat{\varphi}_\lambda(\cdot)\in C^1(W^{1,p(z)}_0(\Omega))$ (see R\u adulescu and Repov\v s \cite[p. 31]{18Rad-Rep}).

\begin{prop}\label{prop6}
  If hypotheses $H_0$, $H_1$ hold and $\lambda>0$, then $\hat{\varphi}_\lambda(\cdot)$ satisfies the $C$-condition.
\end{prop}

\begin{proof}
  We consider a sequence $\{u_n\}_{n\geq1}\subseteq W^{1,p(z)}_0(\Omega)$ such that
\begin{eqnarray} 
  && |\hat{\varphi}_\lambda(u_n)|\leq M_1 \mbox{ for some $M_1>0$, all }n\in\NN, \label{eq9}\\
  && (1+\|u_n\|)\hat{\varphi}'_\lambda(u_n)\to0 \mbox{ in } W^{-1,p'(z)}(\Omega)=W^{1,p(z)}_0(\Omega)^* \mbox{ as }n\to\infty. \label{eq10}
\end{eqnarray}

From \eqref{eq10} we have
\begin{eqnarray}\nonumber 
  \left|\langle \hat{\varphi}'_\lambda(u_n),h\rangle\right| &\leq& \frac{\varepsilon_n\|h\|}{1+\|u_n\|} \mbox{ for all }h\in W^{1,p(z)}_0(\Omega), \mbox{ all }n\in \NN, \mbox{ with }\varepsilon_n\to0^+, \\ \nonumber
  \Rightarrow\Big|\langle A(u_n),h\rangle\!\!&+&\!\!\int_\Omega \xi(z) |u_n|^{p(z)-2}u_n h dz-\int_\Omega\vartheta(u_n^-)^{p(z)-1}hdz \\
   &-&\lambda \int_\Omega(u_n^+)^{q(z)-1}hdz-\int_\Omega f(z,u_n^+)hdz\Big|\leq\frac{\varepsilon_n\|h\|}{1+\|u_n\|} \label{eq11}\\ \nonumber
    &&\mbox{ for all }h\in W^{1,p(z)}_0(\Omega), \ n\in \NN.
\end{eqnarray}

In \eqref{eq11} we choose $h=-u_n^-\in W^{1,p(z)}_0(\Omega)$. We have
\begin{eqnarray}\nonumber 
  && \left|\hat{\rho}(Du_n^-)+\int_\Omega \left[\xi(z)+\vartheta\right](u_n^-)^{p(z)}dz\right|\leq \varepsilon_n \mbox{ for all }n\in\NN, \\
  &\Rightarrow& u_n^-\to0 \mbox{ in } W^{1,p(z)}_0(\Omega) \label{eq12} \\ \nonumber
  && \mbox{ (recall that $\vartheta>\|\xi\|_\infty$ and see Proposition \ref{prop2}(d)). }
\end{eqnarray}

In \eqref{eq11} we choose $h=u_n^+\in W^{1,p(z)}_0(\Omega)$. Then
\begin{equation}\label{eq13}
  \left|\hat{\rho}(Du_n^+)+\int_\Omega\xi(z)(u_n^+)^{p(z)}dz-\lambda\int_\Omega (u_n^+)^{q(z)}dz-\int_\Omega f(z,u_n^+)u_n^+dz\right|\leq \varepsilon_n
\end{equation}
for all $n\in \NN$.

On the other hand, from \eqref{eq9} and \eqref{eq12}, we have
\begin{eqnarray}\nonumber 
|\int_\Omega\frac{p_+}{p(z)}|Du_n^+|^{p(z)}dz+\int_\Omega\frac{p_+}{p(z)}\xi(z)|u_n^+|^{p(z)}dz\!\!\!&-&\!\!\!\lambda\int_\Omega \frac{p_+}{q(z)}(u_n^+)^{q(z)} dz \\
   &-& \int_\Omega p_+F(z,u_n^+)dz|\leq M_2 \label{eq14} \\ \nonumber
  &&\mbox{ for some $M_2>0$, all $n\in\NN$. }
\end{eqnarray}

From \eqref{eq13} and \eqref{eq14} it follows that
\begin{eqnarray}\nonumber 
  && \int_\Omega \left[\frac{p_+}{p(z)}-1\right]|Du_n^+|^{p(z)}dz+\int_\Omega\left[\frac{p_+}{p(z)}-1\right]\xi(z)(u_n^+)^{p(z)}dz \\
  &-& \lambda\int_\Omega \left[\frac{p_+}{q(z)}-1\right](u_n^+)^{q(z)}dz+\int_\Omega e(z,u_n^+)dz\leq M_3 \label{eq15} \\ \nonumber
  && \mbox{ for some } M_3>0, \mbox{ all }n\in\NN.
\end{eqnarray}

Let $\beta_\lambda(z,x)=\lambda\left[1-\frac{p_+}{q(z)}\right]x^{q(z)}+e(z,x)+\xi(z)\left[\frac{p_+}{p(z)}-1\right]x^{p(z)}$ for all $x\geq0$.

Then from \eqref{eq15} we have
\begin{equation}\label{eq16}
  \int_\Omega \beta_\lambda(z,u_n^+)dz\leq M_3 \mbox{ for all }n\in \NN.
\end{equation}

\smallskip
{\it Claim}. The sequence $\{u_n^+\}_{n\geq1}\subseteq W^{1,p(z)}_0(\Omega)$ is bounded.

We argue by contradiction. So, suppose that the claim is not true. Then passing to a subsequence if necessary, we may assume that
\begin{equation}\label{eq17}
  \|u_n^+\|\to\infty \mbox{ as }n\to\infty.
\end{equation}

Let $y_n=\frac{u_n^+}{\|u_n^+\|}$, $n\in\NN$. Then $\|y_n\|=1$, $y_n\geq0$ for all $n\in\NN$. So, we may assume that
\begin{equation}\label{eq18}
  y_n\overset{w}{\to}y \mbox{ in }W^{1,p(z)}_0(\Omega) \mbox{ and } y_n\to y \mbox{ in }L^{p(z)}(\Omega),\;y\geq0.
\end{equation}

Let $\Omega_+=\{z\in\Omega:\:y(z)>0\}$ and $\Omega_0=\{z\in \Omega:\:y(z)=0\}$. Then $\Omega=\Omega_+\cup \Omega_0$ (see \eqref{eq18}).

 First we assume that $|\Omega_+|_N>0$ (by $|\cdot|_N$ we denote the Lebesgue measure on $\RR^N$). We have $u_n^+(z)\to+\infty$ for a.a. $z\in\Omega_+$ and so on account of hypothesis $H_1(ii)$ we have
\begin{eqnarray}\nonumber 
   && \frac{F(z,u_n^+(z))}{u_n^+(z)^{p_+}}\to+\infty \mbox{ for a.a. }z\in\Omega_+, \\ \nonumber
   &\Rightarrow& \frac{F(z,u_n^+(z))}{\|u_n^+\|^{p_+}}=\frac{F(z,u_n^+(z))}{u_n^+(z)^{p_+}}y_n(z)^{p_+}\to+\infty \mbox{ for a.a. }z\in\Omega_+, \\ \nonumber
   &\Rightarrow& \int_{\Omega_+}\frac{F(z,u_n^+)}{\|u_n^+\|^{p_+}}dz\to+\infty \mbox{ (by Fatou's lemma), } \\
   &\Rightarrow& \int_\Omega \frac{F(z,u_n^+)}{\|u_n^+\|^{p_+}}dz\to+\infty \mbox{ as }n\to+\infty. \label{eq19}
\end{eqnarray}

On account of \eqref{eq17}, we may assume that $\|u_n^+\|\geq1$ for all $n\in \NN$. Then from \eqref{eq9} and \eqref{eq12}, we have
\begin{eqnarray}\nonumber 
   && \lambda\int_\Omega \frac{1}{q(z)} \frac{(u_n^+)^{q(z)}}{\|u_n^+\|^{p_+}}dz+\int_\Omega \frac{F(z,u_n^+)}{\|u_n^+\|^{p_+}}dz \\ \nonumber
   &\leq& \varepsilon'_n +\frac{1}{p_-}\int_\Omega |Dy_n|^{p(z)}dz+\frac{\|\xi\|_\infty}{p_-}\int_\Omega y_n^{p(z)}dz \mbox{ with } \varepsilon'_n\to0^+ \\
  &\leq& M_4 \mbox{ for some } M_4>0, \mbox{ all } n\in\NN \mbox{ (see \eqref{eq18}). } \label{eq20}
\end{eqnarray}

Comparing \eqref{eq19} and \eqref{eq20}, we have a contradiction.

So, we assume that $y\equiv0$ (that is, $|\Omega|_N=|\Omega_0|_N$). We define
\begin{equation}\label{eq21}
  \hat{\varphi}_\lambda(t_n u_n)=\max\{\hat{\varphi}_\lambda(tu_n):\:0\leq t\leq1\}.
\end{equation}

Let $v_n=\eta^{\frac{1}{p_-}}y_n$ for all $n\in \NN$, with $\eta>0$. Evidently we have
\begin{equation}\label{eq22}
  v_n\overset{w}{\to}0 \mbox{ in }W^{1,p(z)}_0(\Omega) \mbox{ (see \eqref{eq18}). }
\end{equation}

Hypothesis $H_1(i)$, \eqref{eq18} and the dominated convergence theorem imply that
\begin{equation}\label{eq23}
  \int_\Omega F(z,v_n)dz\to0 \mbox{ as }n\to\infty.
\end{equation}

Also, we have
\begin{eqnarray} 
  && \int_\Omega \frac{1}{p(z)}\xi(z) v_n^{p(z)}dz\to0,\;\int_\Omega \frac{1}{q(z)}v_n^{q(z)}dz\to0 \label{eq24} \\ \nonumber
  && \mbox{ (see \eqref{eq22} and Proposition \ref{prop1}).  }
\end{eqnarray}

Moreover, \eqref{eq17} implies that we can find $n_0\in \NN$ such that
\begin{equation}\label{eq25}
  \frac{\eta^{\frac{1}{p_-}}}{\|u_n^+\|}\in(0,1] \mbox{ for all }n\geq n_0.
\end{equation}

Then from \eqref{eq21} and \eqref{eq25}, we have
\begin{eqnarray*} 
  \hat{\varphi}_\lambda(t_n u_n^+) &\geq& \hat{\varphi}_\lambda(v_n) \\
   &=& \int_\Omega \frac{1}{p(z)}|Dv_n|^{p(z)}dz+\int_\Omega \frac{1}{p(z)} \xi(z) v_n^{p(z)}dz \\
   &-& \lambda\int_\Omega \frac{1}{q(z)}v_n^{q(z)}dz-\int_\Omega F(z,v_n)dz \mbox{ for all }n\geq n_0, \\
   &\geq& \frac{1}{2p_+}\eta \mbox{ for all }n\geq n_1\geq n_0 \\
   && \mbox{ (see \eqref{eq23}, \eqref{eq24} and use the Poincar\'e inequality). }
\end{eqnarray*}

Since $\eta>0$ is arbitrary, we infer that
\begin{equation}\label{eq26}
  \hat{\varphi}_\lambda(t_n u_n^+)\to+\infty \mbox{ as }n\to \infty.
\end{equation}

We know that
\begin{equation}\label{eq27}
  \hat{\varphi}_\lambda(0)=0 \mbox{ and }\hat{\varphi}_\lambda(u_n^+)\leq M_5 \mbox{ for some }M_5>0, \mbox{ all }n\in\NN.
\end{equation}

From \eqref{eq26} and \eqref{eq27} it follows that we can find $n_2\in\NN$ such that
\begin{equation}\label{eq28}
 t_n\in(0,1) \mbox{ for all }n\geq n_2.
\end{equation}

Then from \eqref{eq21} and \eqref{eq28} we infer that
\begin{eqnarray}\nonumber 
  &&  t_n\frac{d}{dt}\hat{\varphi}_\lambda(tu_n^+)|_{t=t_n}=0, \\
  &\Rightarrow& \langle \hat{\varphi}'_\lambda (t_n u_n^+),t_n u_n^+\rangle=0 \mbox{ for all }n\geq n_2 \label{eq29} \\ \nonumber
  && \mbox{ (by the chain rule). }
\end{eqnarray}

For all $n\geq n_2$ we have
\begin{eqnarray}\nonumber 
  && \hat{\varphi}_\lambda(t_n u_n^+) \\ \nonumber
  &=& \hat{\varphi}_\lambda(t_n u_n^+)-\frac{1}{p_+}\langle\hat{\varphi}'_\lambda(t_n u_n^+),t_n u_n^+ \rangle \mbox{ (see \eqref{eq29}) } \\ \nonumber
  &\leq& \int_\Omega \left[\frac{1}{p(z)}-\frac{1}{p_+}\right]|D(t_n u_n^+)|^{p(z)}dz+\int_\Omega\left[\frac{1}{p(z)}-\frac{1}{p_+}\right]\xi(z)(t_n u_n^+)^{p(z)}dz \\ \nonumber
  &-& \lambda\int_\Omega \left[\frac{1}{q(z)}-\frac{1}{p_+}\right](t_nu_n^+)^{q(z)}dz+\frac{1}{p_+}\int_\Omega e(z,t_n u_n^+)dz \\
  &\leq& \int_\Omega \left[\frac{1}{p(z)}-\frac{1}{p_+}\right]|Du_n^+|^{p(z)}dz+\frac{1}{p_+}\int_\Omega \beta_\lambda(z,t_n u_n^+)dz \label{eq30}.
\end{eqnarray}

For the integrand $\beta_\lambda(z,x)$, we have for a.a. $z\in\Omega$, all $x\geq0$
\begin{eqnarray*} 
  (\beta_\lambda)'_x(z,x) &=& \lambda\left[q(z)-p_+\right]x^{q(z)-1}+e'_x(z,x)+\xi(z)\left[p_+-p(z)\right]x^{p(z)-1} \\
  &\geq& \hat{C}x^{p(z)-1}-\lambda C_1 x^{q(z)-1} \mbox{ for some }C_1>0 \\
  && \mbox{ (see hypothesis $H_1(iii)$). }
\end{eqnarray*}

Since $q_+<p_-$, we can find $M_6\geq1$ such that
\begin{eqnarray}\nonumber 
  && (\beta_\lambda)'_x(z,x)\geq0 \mbox{ for a.a. }z\in\Omega, \mbox{ all }x\geq M_6, \\ \nonumber
  &\Rightarrow& \beta_\lambda(z,\cdot) \mbox{ is nondecreasing on }[M_6,\infty) \mbox{ for a.a. }z\in\Omega, \\
  &\Rightarrow& \beta_\lambda(z,x)\leq \beta_\lambda(z,y)+\mu_\lambda(z) \label{eq31} \\ \nonumber
  && \mbox{ for a.a. $z\in\Omega$, all $0\leq x \leq y$, with $\mu_\lambda\in L^1(\Omega)$. }
\end{eqnarray}

From \eqref{eq30} and \eqref{eq31} it follows that
\begin{eqnarray*} 
  && \hat{\varphi}_\lambda(t_n u_n^+) \\
  &\leq& \int_\Omega\left[\frac{1}{p(z)}-\frac{1}{p_+}\right]|Du_n^+|^{p(z)}dz+\frac{1}{p_+}\int_\Omega \beta_\lambda (z,u_n^+)dz+\frac{1}{p}\|\mu_\lambda\|_1 \\
  && \mbox{ for all }n\geq n_2 \\
  &=& \hat{\varphi}_\lambda(u_n^+)-\frac{1}{p_+}\langle \hat{\varphi}'_\lambda(u_n^+),u_n^+\rangle\\
  &\leq& \hat{\varphi}_\lambda(u_n^+)+\frac{\varepsilon_n}{p_+} \mbox{ for all }n\geq n_2 \mbox{ (see \eqref{eq13}), } \\
  &\Rightarrow& \hat{\varphi}_\lambda(u_n^+)\to+\infty \mbox{ as }n\to\infty \mbox{ (see \eqref{eq26}), } \\
  && \mbox{ a contradiction (see \eqref{eq27}). }
\end{eqnarray*}

Therefore $\{u_n^+\}\subseteq W^{1,p(z)}_0(\Omega)$ is bounded and this proves the claim.

Then from \eqref{eq12} and the claim it follows that
$$
\{u_n\}_{n\geq1}\subseteq W^{1,p(z)}_0(\Omega) \mbox{ is bounded. }
$$

We may assume that
\begin{equation}\label{eq32}
  u_n\overset{w}{\to}u \mbox{ in }W^{1,p(z)}_0(\Omega) \mbox{ and }u_n\to u \mbox{ in }L^{r(z)}(\Omega) \mbox{ as }n\to\infty.
\end{equation}

In \eqref{eq11} we choose $h=u_n-u\in W^{1,p(z)}_0(\Omega)$, pass to the limit as $n\to\infty$ and use \eqref{eq32}. Then
\begin{eqnarray*} 
  && \lim_{n\to\infty} \langle A(u_n),u_n-u\rangle=0, \\
  &\Rightarrow& u_n\to u \mbox{ in } W^{1,p(z)}_0(\Omega) \mbox{ (see Proposition \ref{prop3}), } \\
  &\Rightarrow& \hat{\varphi}_\lambda(\cdot) \mbox{ satisfies the $C$-condition. }
\end{eqnarray*}
The proof is now complete.
\end{proof}

\begin{prop}
  If hypotheses $H_0$, $H_1$ hold, then $\mathcal{L}\not=\emptyset$ and we have $S_\lambda\subseteq{\rm int}\,C_+$ for every $\lambda\in\mathcal{L}$.
\end{prop}

\begin{proof}
  On account of hypotheses $H_1(i),(iv)$, we see that given $\varepsilon>0$, we can find $C_2=C_2(\varepsilon)>0$ such that
\begin{equation}\label{eq33}
  F(z,x)\leq\frac{\varepsilon}{p_+}x^{p_+}+C_2 x^{r_+} \mbox{ for a.a. }z\in\Omega, \mbox{ all }x\geq0.
\end{equation}

For every $u\in W^{1,p(z)}_0(\Omega)$, we have
\begin{eqnarray}\nonumber 
  \hat{\varphi}_\lambda(u) &\geq& \int_\Omega \frac{1}{p(z)}|Du|^{p(z)}dz+\int_\Omega \frac{1}{p(z)}\xi(z)|u|^{p(z)}dz \\
  &-& \frac{\lambda}{q_+}\int_\Omega (u^+)^{q(z)}dz-\frac{\varepsilon}{p_+}\|u\|_{p_+}^{p_+}-C_3\|u\|^{r_+} \label{eq34} \\ \nonumber
  && \mbox{ for some $C_3>0$ (see \eqref{eq33} and recall that $\vartheta>\|\xi\|_\infty$).}
\end{eqnarray}

For $\|u\|_{p(z)}\leq 1$ we have
\begin{eqnarray}\nonumber 
   && \int_\Omega \frac{1}{p(z)}\xi(z)|u|^{p(z)}dz\leq\frac{\|\xi\|_\infty}{p_-}\|u\|_{p(z)}^{p_-}\leq C_4\|u\|^{p_-} \mbox{ for some }C_4>0, \\
   &\Rightarrow& \int_\Omega \frac{1}{p(z)} \xi(z)(u^+)dz\leq \lambda\|u\|^{q_+}+C_5\|u\|^{r_+} \label{eq35} \\ \nonumber
   && \mbox{ for some $C_5=C_5(\lambda)>0$ (recall that $q_+<p_-\leq p_+<r_+$). }
\end{eqnarray}

We return to \eqref{eq34} and use \eqref{eq35}. Then for $u\in W^{1,p(z)}_0(\Omega)$ with $\max\{\|u\|,\|u\|_{p(z)}\}\leq1$ we have
\begin{eqnarray*} 
  \hat{\varphi}_\lambda(u)&\geq& \frac{1}{p_+}\left(1-\varepsilon C_6\right)\|u\|^{p_+}-C_7\left(\lambda\|u\|^{q_+}+\|u\|^{r_+}\right)  \\
  && \mbox{ for some }C_6,\, C_7>0.
\end{eqnarray*}

We choose $\varepsilon\in\left(0,\frac{1}{C_6}\right)$ and obtain
\begin{eqnarray}\nonumber 
   && \hat{\varphi}_\lambda(u)\geq C_8\|u\|^{p_+}-C_7\left(\lambda\|u\|^{q_+}+\|u\|^{r_+}\right) \mbox{ for some }C_8>0, \\
  &\Rightarrow& \hat{\varphi}_\lambda(u)\geq\left[C_8-C_7\left(\lambda\|u\|^{q_+-p_+}+\|u\|^{r_+-p_+}\right)\right]\|u\|^{p_+}. \label{eq36}
\end{eqnarray}

Consider the function
$$
k_\lambda(t)=\lambda t^{q_+-p_+}+t^{r_+-p_+} \mbox{ for all }t>0.
$$

Evidently, $k_\lambda\in C^1(0,+\infty)$ and since $q_+<p_-\leq p_+<r_+$ we have
$$
k_\lambda(t)\to+\infty \mbox{ as }t\to0^+ \mbox{ and as }t\to+\infty.
$$

So, we can find $t_0>0$ such that
\begin{eqnarray}\nonumber 
   && k_\lambda(t_0)=\min\{k_\lambda(t):\: t>0\}, \\ \nonumber
   &\Rightarrow& k'_\lambda(t_0)=0, \\ \nonumber
   &\Rightarrow& \lambda(p_+-q_+)t_0^{q_+-p_+-1}=(r_+ -p_+)t_0^{r_+-p_+-1}, \\
   &\Rightarrow& t_0=\left[\frac{\lambda(p_+-q_+)}{r_+-p_+}\right]^{\frac{1}{r_+-q_+}}. \label{eq37}
\end{eqnarray}

Then
\begin{eqnarray*} 
  && k_\lambda(t_0)=\lambda^{\frac{r_+-p_+}{r_+-q_+}}
\frac{(r_+-p_+)^{\frac{p_+-q_+}{r_+-q_+}}}{(p_+-q_+)^{\frac{p_+-q_+}{r_+-q_+}}}+\lambda^{\frac{r_+-p_+}{r_+-q_+}}
\frac{(p_+-q_+)^{\frac{r_+-p_+}{r_+-q_+}}}{(r_+-p_+)^{\frac{r_+-p_+}{r_+-q_+}}}, \\
  &\Rightarrow& k_\lambda(t_0)\to0 \mbox{ as }\lambda\to 0^+.
\end{eqnarray*}

Let $C_0>0$ be such that $\|\cdot\|_{p(z)}\leq C_0\|\cdot\|$. So, we can find $\lambda_0>0$ such that
$$
0<t_0\leq \min\left\{\frac{1}{C_0},1\right\} \mbox{ and } k_\lambda(t_0)<\frac{C_8}{C_7} \mbox{ for all } \lambda\in(0,\lambda_0) \mbox{ (see \eqref{eq36}, \eqref{eq37}). }
$$

Then from \eqref{eq36} it follows that
\begin{equation}\label{eq38}
  \hat{\varphi}_\lambda(u)\geq\hat{m}_\lambda>0 \mbox{ for all }\|u\|=t_0.
\end{equation}

On account of superlinearity hypothesis $H_1(ii)$, for $u\in {\rm int}\,C_+$, we have
\begin{equation}\label{eq39}
  \hat{\varphi}_\lambda(tu)\to-\infty \mbox{ as }t\to+\infty.
\end{equation}

Then \eqref{eq38}, \eqref{eq39} and Proposition \ref{prop6}, permit the use of the mountain pass theorem. Therefore for every $\lambda\in(0,\lambda_0)$ we can find $u_\lambda\in W^{1,p(z)}_0(\Omega)$ such that
\begin{equation}\label{eq40}
  u_\lambda\in K_{\hat{\varphi}_\lambda} \mbox{ and }0<\hat{m}_\lambda\leq\hat{\varphi}_\lambda(u_\lambda) \mbox{ (see \eqref{eq38}). }
\end{equation}

From \eqref{eq40} we have $u_\lambda\not=0$ (recall that $\hat{\varphi}_\lambda(0)=0$) and
\begin{equation}\label{eq41}
  \langle \hat{\varphi}'_\lambda(u_\lambda),h\rangle=0 \mbox{ for all }h\in W^{1,p(z)}_0(\Omega).
\end{equation}

Choosing $h=-u_\lambda^-\in W^{1,p(z)}_0(\Omega)$, we obtain
\begin{eqnarray*} 
  && \int_\Omega \frac{1}{p(z)}|Du_\lambda^-|^{p(z)}dz+\int_\Omega \frac{\vartheta+\xi(z)}{p(z)}(u_\lambda^-)^{p(z)}dz=0, \\
  &\Rightarrow& \frac{1}{p_+}\left[\hat{\rho}(Du_\lambda^-)+C_9\rho(u_\lambda^-)\right]\leq0 \mbox{ for some }C_9>0, \\
  &\Rightarrow& u_\lambda\geq0,\:u_\lambda\not=0.
\end{eqnarray*}

Then from \eqref{eq41} it follows that $u_\lambda$ is a positive solution $(P_\lambda)$. As before the anisotropic regularity theory (see \cite{5Fan}, \cite{6Fan-Zha}) implies that
$$
u_\lambda\in C_+\setminus\{0\}.
$$

We have
\begin{eqnarray*} 
  && -\Delta_{p(z)}u(z)+\xi(z)u(z)^{p(z)-1}\geq0 \mbox{ for a.a. }z\in\Omega, \\
  &\Rightarrow& \Delta_{p(z)}u(z)\leq\|\xi\|_\infty u(z)^{p(z)-1} \mbox{ for a.a. }z\in\Omega, \\
  &\Rightarrow& u\in{\rm int}\,C_+ \mbox{ (see Zhang \cite{21Zha}). }
\end{eqnarray*}

So, we have proved that $(0,\lambda_0)\subseteq\mathcal{L}$ and so $\mathcal{L}\not=\emptyset$. Moreover, we have $S_\lambda\subseteq{\rm int}\,C_+$ for all $\lambda>0$.
\end{proof}

Next, we show that $\mathcal{L}$ is an interval.

\begin{prop}\label{prop8}
  If hypotheses $H_0$, $H_1$ hold, $\lambda\in\mathcal{L}$ and $0<\mu<\lambda$, then $u\in \mathcal{L}$ and given $u_\lambda\in S_\lambda$ we can find $u_\mu\in S_\mu$ such that $u_\mu\leq u_\lambda$.
\end{prop}

\begin{proof}
  Since $\lambda\in \mathcal{L}$, we can find $u_\lambda\in S_\lambda\subseteq{\rm int}\,C_+$. With $\vartheta>\|\xi\|_\infty$, we introduce the Carath\'eodory function $g_\mu(z,x)$ defined by
\begin{equation}\label{eq42}
  g_\mu(z,x)=\left\{
               \begin{array}{ll}
                 \mu(x^+)^{q(z)-1}+f(z,x^+)+\vartheta(x^+)^{p(z)-1}, & \hbox{ if }x\leq u_\lambda(z) \\
                 \mu u_\lambda(z)^{q(z)-1}+f(z,u_\lambda(z))+\vartheta u_\lambda(z)^{p(z)-1}, & \hbox{ if } u_\lambda(z)<x.
               \end{array}
             \right.
\end{equation}

We set $G_\mu(z,x)=\displaystyle{\int_0^x g_\mu(z,s)ds}$ and consider the $C^1$-functional $\Psi_\mu: W^{1,p(z)}_0(\Omega)\to\RR$ defined by
$$
\Psi_\mu(u)=\int_\Omega \frac{1}{p(z)}|Du|^{p(z)}dz+\int_\Omega \frac{\vartheta+\xi(z)}{p(z)}|u|^{p(z)}dz-\int_\Omega G_\mu(z,u)dz
$$
for all $u\in W^{1,p(z)}_0(\Omega)$.

Since $\vartheta>\|\xi\|_\infty$, from \eqref{eq42} it is clear that $\Psi_\mu(\cdot)$ is coercive. Also using the fact that $W^{1,p(z)}_0(\Omega)\hookrightarrow L^{p(z)}(\Omega)$ compactly, we see  that $\Psi_\mu(\cdot)$ is sequentially weakly lower semicontinuous. So, by the Weierstrass-Tonelli theorem, there exists $u_\mu\in W^{1,p(z)}_0(\Omega)$ such that
\begin{equation}\label{eq43}
  \Psi_\mu(u_\mu)=\inf\left\{\Psi_\mu(u):\: u\in W^{1,p(z)}_0(\Omega)\right\}.
\end{equation}

Since $q_+<p_-$, we see that
\begin{eqnarray*} 
  && \Psi_\mu(u_\mu)<0=\Psi_\mu(0), \\
  &\Rightarrow& u_\mu\not=0.
\end{eqnarray*}

From \eqref{eq43} we have
$$
\Psi'_\mu(u_\mu)=0,
$$
\begin{eqnarray} 
  &\Rightarrow& \langle A(u_\mu),h\rangle+\int_\Omega\left[\vartheta+\xi(z)\right]|u_\mu|^{p(z)-2}u_\mu hdz=\int_\Omega g_\mu(z,u_\mu)hdz  \label{eq44}\\ \nonumber
  && \mbox{ for all }h\in W^{1,p(z)}_0(\Omega).
\end{eqnarray}

In \eqref{eq44} first we choose $h=-u_\mu^-\in W^{1,p(z)}_0(\Omega)$. We obtain
\begin{eqnarray*} 
  && \hat{\rho}(Du_\mu^-)+C_{10}\rho(u_\mu^-)\leq0 \mbox{ for some } C_{10}>0 \mbox{ (see \eqref{eq42}), } \\
  &\Rightarrow& u_\mu\geq0,\;u_\mu\not=0.
\end{eqnarray*}

Next, in \eqref{eq44} we choose $h=(u_\mu-u_\lambda)^+\in W^{1,p(z)}_0(\Omega)$. We have
\begin{eqnarray*} 
  && \langle A(u_\mu),(u_\mu-u_\lambda)^+\rangle +\int_\Omega [\vartheta+\xi(z)]u_\mu^{p(z)-1}(u_\mu-u_\lambda)^+dz \\
  &=& \int_\Omega [\mu u_\lambda^{q(z)-1}+f(z,u_\lambda)+\vartheta u_\lambda ^{p(z)-1}](u_\mu-u_\lambda)^+dz \mbox{ (see \eqref{eq42}) } \\
  &\leq& \int_\Omega[\lambda u_\lambda u^{q(z)-1}+f(z,u_\lambda)+\vartheta u_\lambda^{p(z)-1}](u_\mu-u_\lambda)^+dz \mbox{ (since $\mu<\lambda$) } \\
  &=& \langle A(u_\lambda),(u_\mu-u_\lambda)^+\rangle+\int_\Omega [\vartheta+\xi(z)]u_\lambda^{p(z)-1}(u_\mu-u_\lambda)^+dz \mbox{ (since $u_\lambda\in S_\lambda$). }
\end{eqnarray*}

The monotonicity of $A(\cdot)$ (see Proposition \ref{prop3}) and the fact that $\vartheta>\|\xi\|_\infty$ imply that
\begin{eqnarray*} 
  && u_\mu\leq u_\lambda, \\
  &\Rightarrow& u_\mu\in[0,u_\lambda],\;u_\mu\not=0, \\
  &\Rightarrow& u_\mu\in S_\mu\subseteq{\rm int}\,C_+ \mbox{ (see \eqref{eq42} and \eqref{eq44}). }
\end{eqnarray*}
The proof is now complete.
\end{proof}

So, according to Proposition \ref{prop8} the solution multifunction $\lambda\mapsto S_\lambda$ has a kind of weak monotonicity property. We can improve this monotonicity property by adding one more condition on the perturbation $f(z,\cdot)$.

The new hypotheses on $f(z,x)$ are the following:

\smallskip
$H_2$: $f:\Omega\times\RR\to \RR$ is a function which is measurable in $z\in\Omega$, for a.a. $z\in \Omega$ we have $f(z,\cdot)\in C^1(\RR)$, hypotheses $H_2(i)\to(iv)$ are the same as the corresponding hypotheses $H_1(i)\to(iv)$, and
\begin{itemize}
  \item[(v)] for every $\rho>0$, there exists $\hat{\xi}_\rho>0$ such that for a.a. $z\in\Omega$ the function
$$
x\mapsto f(z,x)+\hat{\xi}_\rho x^{p(z)-1}
$$
is nondecreasing on $[0,\rho]$.
\end{itemize}

\begin{rem}\label{rem3}
  This is a one-sided local H\"older condition on $f(z,\cdot)$. It is satisfied if for every $\rho>0$, we can find $\hat{C}_\rho>0$ such that $f'_x(z,x)\geq-\hat{C}_\rho x^{p(z)-1}$ for a.a. $z\in\Omega$, all $0\leq x\leq \rho$.
\end{rem}

\begin{prop}\label{prop9}
  If hypotheses $H_0$, $H_2$ hold, $\lambda\in\mathcal{L}$, $u_\lambda\in S_\lambda\subseteq{\rm int}\,C_+$ and $\mu\in (0,\lambda)$, then $\mu\in\mathcal{L}$ and we can find $u_\mu\in S_\mu\subseteq{\rm int}\,C_+$ such that
$$
u_\lambda-u_\mu\in{\rm int}\,C_+.
$$
\end{prop}

\begin{proof}
  From Proposition \ref{prop8} we know that $\mu\in \mathcal{L}$ and there exists $u_\mu\in S_\mu\subseteq {\rm int}\,C_+$ such that
\begin{equation}\label{eq45}
  u_\lambda-u_\mu\in C_+\setminus\{0\}.
\end{equation}

Let $\rho=\|u_\lambda\|_\infty$ and let $\hat{\xi}_\rho>0$ be as postulated by hypothesis $H_2(v)$. We can always assume that $\hat{\xi}_\rho>\|\xi\|_\infty$. Then we have
\begin{eqnarray}\nonumber 
  && -\Delta_{p(z)} u_\mu+[\xi(z)+\hat{\xi}_\rho]u_\mu^{p(z)-1} \\ \nonumber
  &=& \mu u_\mu^{q(z)-1}+f(z,u_\mu)+\hat{\xi}_\rho u_\mu^{p(z)-1} \\ \nonumber
  &\leq& \mu u_\lambda^{q(z)-1}+f(z,u_\lambda)+\hat{\xi}_\rho u_\lambda^{p(z)-1} \mbox{ (see \eqref{eq45} and hypothesis $H_2(v)$) } \\ \nonumber
  &\leq& \lambda u_\lambda^{q(z)-1}+f(z,u_\lambda)+\hat{\xi}_\rho u_\lambda^{p(z)-1} \mbox{ (since $\mu<\lambda$) } \\
  &=&-\Delta_{p(z)}u_\lambda+[\xi(z)+\hat{\xi}_\rho]u_\lambda^{p(z)-1}. \label{eq46}
\end{eqnarray}

Note that since $u_\lambda\in{\rm int}\,C_+$ and $\mu<\lambda$, we have
\begin{equation}\label{eq47}
  0\prec(\lambda-\mu)u_\lambda^{q(z)-1}.
\end{equation}

Then from \eqref{eq46}, \eqref{eq47} and Proposition \ref{prop4}, we conclude that
$$
u_\lambda-u_\mu\in{\rm int}\,C_+.
$$
The proof is now complete.
\end{proof}

Next, we show that for every $\lambda\in \mathcal{L}$, the solution set $S_\lambda$ has a smallest element (minimal positive solution).

To this end, first we consider the following auxiliary problem
\begin{equation}\label{eq48}
\left\{
\begin{array}{lll}
-\Delta_{p(z)}u(z)+|\xi(z)| |u(z)|^{p(z)-2}u(z)=\lambda |u(z)|^{q(z)-2}u(z) \text{ in } \Omega,\\
u|_{\partial\Omega}=0,\;\lambda>0,u>0.
\end{array}
\right.
\end{equation}

\begin{prop}\label{prop10}
  If hypotheses $H_0$ hold and $\lambda>0$, then problem \eqref{eq48} admits a unique positive solution $\overline{u}_\lambda\in{\rm int}\,C_+$.
\end{prop}

\begin{proof}
We consider the $C^1$-functional $\gamma_\lambda: W^{1,p(z)}_0(\Omega)\to\RR$ defined by
$$
\gamma_\lambda(u)=\int_\Omega \frac{1}{p(z)}|Du|^{p(z)}dz+\int_\Omega \frac{|\xi(z)|}{p(z)}|u|^{p(z)}dz-\lambda\int_\Omega \frac{1}{q(z)}(u^+)^{q(z)}dz
$$
for all $u\in W^{1,p(z)}_0(\Omega)$.

Evidently, $\gamma_\lambda(\cdot)$ is coercive (since $q_+<p_-$) and sequentially weakly lower semicontinuous. So, we can find $\overline{u}_\lambda\in W^{1,p(z)}_0(\Omega)$ such that
\begin{eqnarray*} 
  && \gamma_\lambda(\overline{u}_\lambda)=\min\left\{\gamma_\lambda(u):\:u\in W^{1,p(z)}_0(\Omega)\right\}<0=\gamma_\lambda(0) \mbox{ (since $q_+<p_-$), } \\
  &\Rightarrow& \overline{u}_\lambda\not=0.
\end{eqnarray*}

We have
$$
\gamma'_\lambda(\overline{u}_\lambda)=0,
$$
\begin{equation}\label{eq49}
  \Rightarrow\langle A(\overline{u}_\lambda),h\rangle+\int_\Omega |\xi(z)| |\overline{u}_\lambda|^{p(z)-2}\overline{u}_\lambda hdz=\lambda\int_\Omega (\overline{u}_\lambda^+)^{q(z)-1}hdz
\end{equation}
for all $h\in W^{1,p(z)}_0(\Omega)$.

In \eqref{eq49} we choose $h=-\overline{u}_\lambda^-\in W^{1,p(z)}_0(\Omega)$. Then
\begin{eqnarray*} 
   && \hat{\rho}(D\overline{u}_\lambda^-)+\int_\Omega |\xi(z)| (\overline{u}_\lambda^-)^{p(z)}dz=0, \\
   &\Rightarrow& \overline{u}_\lambda\geq 0,\;\overline{u}_\lambda\not=0, \\
   &\Rightarrow& \overline{u}_\lambda \mbox{ is a positive solution of \eqref{eq48} (see \eqref{eq49}), } \\
   &\Rightarrow& \overline{u}_\lambda\in C_+\setminus\{0\} \mbox{ (anisotropic regularity theory). }
\end{eqnarray*}

Therefore
\begin{eqnarray*} 
  && \Delta_{p(z)}\overline{u}_\lambda(z)\leq \|\xi\|_\infty \overline{u}_\lambda(z)^{p(z)-1} \mbox{ for a.a. }z\in \Omega, \\
  &\Rightarrow& \overline{u}_\lambda\in{\rm int}\,C_+ \mbox{ (see Zhang \cite{21Zha}). }
\end{eqnarray*}

Next, we show that this positive solution of \eqref{eq48} in unique.

Suppose that $\overline{v}_\lambda$ is another positive solution of \eqref{eq48}. Again we have $\overline{v}_\lambda\in{\rm int}\,C_+$. On account of Proposition 4.1.22 of Papageorgiou, R\u adulescu and Repov\v s \cite[p. 274]{17Pap-Rad-Rep}, we have $\frac{\overline{u}_\lambda}{\overline{v}_\lambda},\frac{\overline{v}_\lambda}{\overline{u}_\lambda}\in L^\infty(\Omega)$. So, we can apply Theorem 2.5 of Taka\v c and Giacomoni \cite{19Tak-Gia} and have
\begin{eqnarray*} 
  0 &\leq& \int_\Omega \left[\frac{-\Delta_{p(z)}\overline{u}_\lambda}{\overline{u}_\lambda^{p_- -1}}+\frac{-\Delta_{p(z)}\overline{v}_\lambda}{\overline{v}_\lambda^{p_- -1}}\right](\overline{u}_\lambda^{p_-}-\overline{v}_\lambda^{p_-})dz \\
  &=& \int_\Omega \left[\lambda\left(\overline{u}_\lambda^{q(z)-p_-}-\overline{v}_\lambda^{q(z)-p_-}\right)-|\xi(z)|
\left(\overline{u}_\lambda^{p(z)-p_-}-\overline{v}_\lambda^{p(z)-p_-}\right)\right](\overline{u}_\lambda^{p_-}-
\overline{v}_\lambda^{p_-})dz, \\
  \Rightarrow \overline{u}_\lambda &=& \overline{v}_\lambda \mbox{ (since $q_+<p_-\leq p(z)$). }
\end{eqnarray*}

Therefore the positive solution $\overline{u}_\lambda\in{\rm int}\,C_+$ of problem \eqref{eq48} is unique.
\end{proof}

This solution $\overline{u}_\lambda\in{\rm int}\,C_+$ provides a lower bound for the solution set $S_\lambda$.

\begin{prop}\label{prop11}
  If hypotheses $H_0$, $H_1$ hold and $\lambda\in\mathcal{L}$, then $\overline{u}_\lambda\leq u$ for all $u\in S_\lambda$.
\end{prop}

\begin{proof}
  Let $u\in S_\lambda\subseteq{\rm int}\,C_+$ and consider the Carath\'eodory function $\beta_\lambda(z,x)$ defined by
\begin{equation}\label{eq50}
  \beta_\lambda(z,x)=\left\{
                       \begin{array}{ll}
                         \lambda(x^+)^{q(z)-1}, & \hbox{ if }x\leq u(z) \\
                         \lambda u(z)^{q(z)-1}, & \hbox{ if } u(z)<x.
                       \end{array}
                     \right.
\end{equation}

We set $B_\lambda(z,x)=\displaystyle{\int_0^x \beta_\lambda(z,s)ds}$ and consider the $C^1$-functional $\tau_\lambda:W^{1,p(z)}_0(\Omega)\to\RR$ defined by
$$
\tau_\lambda(u)=\int_\Omega \frac{1}{p(z)}|Du|^{p(z)}dz+\int_\Omega \frac{|\xi(z)|}{p(z)}|u|^{p(z)}dz-\int_\Omega B_\lambda(z,u)dz
$$
for all $u\in W^{1,p(z)}_0(\Omega)$.

From \eqref{eq50} we see that $\tau_\lambda(\cdot)$ is coercive. Also, it is sequentially weakly lower semicontinuous. So, we can find $\tilde{u}_\lambda\in W^{1,p(z)}_0(\Omega)$ such that
\begin{eqnarray*} 
  && \tau_\lambda(\tilde{u}_\lambda)=\min\left\{\tau_\lambda(u):\: u\in W^{1,p(z)}_0(\Omega)\right\}<0=\tau_\lambda(0) \mbox{ (since $q_+<p_-$), } \\
  &\Rightarrow& \tilde{u}_\lambda\not=0.
\end{eqnarray*}

We have
$$
\tau'_\lambda(\tilde{u})=0,
$$
\begin{equation}\label{eq51}
  \Rightarrow\langle A(\tilde{u}_\lambda),h\rangle+\int_\Omega|\xi(z)| |\tilde{u}_\lambda|^{p(z)-2}\tilde{u}_\lambda hdz=\int_\Omega\beta_\lambda(z, \tilde{u}_\lambda)hdz
\end{equation}
for all $h\in W^{1,p(z)}_0(\Omega)$.

In \eqref{eq51} first we choose $h=-\tilde{u}_\lambda^-\in W^{1,p(z)}_0(\Omega)$ and infer that
$$
\tilde{u}_\lambda\geq0,\;\tilde{u}_\lambda\not=0.
$$

Next, in \eqref{eq51} we choose $h=(\tilde{u}_\lambda-u)^+\in W^{1,p(z)}_0(\Omega)$. We have
\begin{eqnarray*} 
   && \langle A(\tilde{u}_\lambda),(\tilde{u}_\lambda-u)^+\rangle+\int_\Omega |\xi(z)|\tilde{u}_\lambda^{p(z)-1}(\tilde{u}_\lambda-u)^+dz \\
   &=& \int_\Omega \lambda u^{q(z)-1}(\tilde{u}_\lambda-u)^+dz \mbox{ (see \eqref{eq50}) } \\
   &\leq& \int_\Omega \left[\lambda u^{q(z)-1}+f(z,u)\right](\tilde{u}_\lambda-u)^+dz \mbox{ (since $f\geq0$) } \\
   &\leq& \langle A(u),(\tilde{u}_\lambda-u)^+\rangle+\int_\Omega |\xi(z)|u^{p(z)-1}(\tilde{u}_\lambda-u)^+dz \mbox{ (since $u\in S_\lambda$), }\\
   \Rightarrow \tilde{u}_\lambda&\leq& u.
\end{eqnarray*}

So, we have proved that
\begin{equation}\label{eq52}
  \tilde{u}_\lambda\in[0,u]\setminus\{0\}.
\end{equation}

Then from \eqref{eq51}, \eqref{eq52} and \eqref{eq50} it follows that
\begin{eqnarray*} 
  && \tilde{u}_\lambda \mbox{ is a positive solution of \eqref{eq48}, } \\
  &\Rightarrow& \tilde{u}_\lambda=\overline{u}_\lambda\in{\rm int}\,C_+ \mbox{ (see Proposition \ref{prop10}), } \\
  &\Rightarrow& \overline{u}_\lambda\leq u \mbox{ for all }u\in S_\lambda.
\end{eqnarray*}
The proof is now complete.
\end{proof}

\begin{rem}\label{rem4}
  Reasoning as in the above proof, we show that $\lambda\mapsto \overline{u}_\lambda$ is increasing that is, if $0<\mu<\lambda$, then $\overline{u}_\lambda-\overline{u}_\mu\in C_+\setminus\{0\}$.
\end{rem}

We know that $S_\lambda$ is downward directed (see Filippakis and Papageorgiou \cite{7Fil-Pap} and Papageorgiou, R\u adulescu and Repov\v s \cite{16Pap-Rad-Rep} and recall that $A(\cdot)$ is monotone, see Proposition \ref{prop3}).
\begin{prop}\label{prop12}
  If hypotheses $H_0$, $H_1$ hold and $\lambda\in \mathcal{L}$, then there exists $u_\lambda^*\in S_\lambda\subseteq{\rm int}\,C_+$ such that
$$
u_\lambda^*\leq u \mbox{ for all }u\in S_\lambda
$$
$$
\mbox{ (minimal positive solution of $(P_\lambda)$). }
$$
\end{prop}

\begin{proof}
   By Lemma 3.10 of Hu and Papageorgiou \cite[p. 178]{12Hu-Pap}, we know that we can find $\{u_n\}_{n\geq1}\subseteq S_\lambda\subseteq{\rm int}\,C_+$ decreasing (recall that $S_\lambda$ is downward directed) such that
$$
\inf_{n\geq1} u_n=\inf S_\lambda.
$$

Since $\overline{u}_\lambda\leq u_n\leq u_1$ for all $n\in\NN$ (see Proposition \ref{prop11}), from hypothesis $H_1(i)$ it follows that
$$
\{u_n\}_{n\geq1}\subseteq W^{1,p(z)}_0(\Omega) \mbox{ is bounded. }
$$

So, we may assume that
\begin{equation}\label{eq53}
  u_n\overset{w}{\to} u_\lambda^* \mbox{ in }W^{1,p(z)}_0(\Omega) \mbox{ and } u_n\to u_\lambda^* \mbox{ in }L^{r(z)}(\Omega) \mbox{ as }n\to\infty.
\end{equation}

We have
\begin{eqnarray}\nonumber 
   && \langle A(u_n),u_n-u_\lambda^*\rangle+\int_\Omega \xi(z)u_n^{p(z)-1}(u_n-u_\lambda^*)dz \\ \nonumber
  &=& \lambda\int_\Omega u_n^{q(z)-1}(u_n-u_\lambda^*)dz+\int_\Omega f(z,u_n)(u_n-u_\lambda^*)dz, \\ \nonumber
  &\Rightarrow& \lim_{n\to\infty}\langle A(u_n),u_n-u_\lambda^*\rangle=0, \\
  &\Rightarrow& u_n\to u_\lambda^* \mbox{ in } W^{1,p(z)}_0(\Omega) \mbox{ (see Proposition \ref{prop3}).} \label{eq54}
\end{eqnarray}

Note that
$$
\overline{u}_\lambda\leq u_\lambda^* \mbox{ and so }u_\lambda^*\not=0,
$$
\begin{eqnarray*} 
  && \langle A(u_\lambda^*),h\rangle+\int_\Omega \xi(z)(u_\lambda^*)^{p(z)-1}hdz=\lambda\int_\Omega (u_\lambda^*)^{q(z)-1}hdz +\int_\Omega f(z,u_\lambda^*)hdz \\
  && \mbox{ for all } h\in W^{1,p(z)}_0(\Omega) \mbox{ (see \eqref{eq54}). }
\end{eqnarray*}

It follows that
$$
u_\lambda^*\in S_\lambda\subseteq {\rm int}\,C_+ \mbox{ and }u_\lambda^*=\inf S_\lambda.
$$
The proof is now complete.
\end{proof}

We set $\lambda^*=\sup \mathcal{L}$.

\begin{prop}\label{prop13}
  If hypotheses $H_0$, $H_2$ hold, then $\lambda^*<\infty$.
\end{prop}

\begin{proof}
  On account of hypotheses $H_0$, $H_2(iv)$ and since $q_+<p_-$, we see that we can find $\hat{\lambda}>0$ such that
\begin{equation}\label{eq55}
  \hat{\lambda}x^{q(z)-1}+f(z,x)-\xi(z)x^{p(z)-1}\geq0 \mbox{ for a.a. }z\in\Omega, \mbox{ all }x\geq0.
\end{equation}

Let $\lambda>\hat{\lambda}$ and suppose that $\lambda\in \mathcal{L}$. Then we can find $u_\lambda\in S_\lambda\subseteq{\rm int}\,C_+$. Let $\Omega_0\subset\subset \Omega$ (that is, $\Omega_0\subseteq\overline{\Omega}_0\subseteq\Omega$) and assume that $\partial\Omega_0$  is a $C^2$-manifold. We set $\displaystyle{m_0=\min_{\overline{\Omega}_0} u_\lambda>0}$ (recall that $u_\lambda\in{\rm int}\,C_+$). Also, let $\hat{\xi}_\rho>\|\xi\|_\infty$. Let $m_0^\delta=m_0+\delta$ for $\delta>0$ small. We have
\begin{eqnarray}\nonumber 
  && -\Delta_{p(z)}m_0^\delta+[\xi(z)+\hat{\xi}_\rho](m_0^\delta)^{p(z)-1} \\ \nonumber
  &\leq& [\xi(z)+\hat{\xi}_\rho]m_0^{p(z)-1}+\chi(\delta) \mbox{ with }\chi(\delta)\to0^+ \mbox{ as }\delta\to0^+\\ \nonumber
  &\leq& \hat{\lambda}m_0^{q(z)-1}+f(z,m_0)+\hat{\xi}_\rho m_0^{p(z)-1}+\chi(\delta) \mbox{ (see \eqref{eq55}) } \\ \nonumber
  &\leq& \hat{\lambda}u_\lambda^{q(z)-1}+f(z,u_\lambda)+\hat{\xi}_\rho u_\lambda^{p(z)-1}+\chi(\delta) \mbox{ (see hypothesis $H_2(iv)$) } \\ \nonumber
  &\leq& \lambda u_\lambda^{q(z)-1}+f(z,u_\lambda)+\hat{\xi}_\rho u_\lambda^{p(z)-1}-[\lambda-\hat{\lambda}]m_0^{p(z)-1}+\chi(\delta) \\
  &\leq& -\Delta_{p(z)} u_\lambda+[\xi(z)+\hat{\xi}_\rho] u_\lambda^{p(z)-1} \mbox{ in }\Omega_0 \mbox{ for }\delta\in(0,1) \mbox{ small}. \label{eq56}
\end{eqnarray}

Note that for $\delta\in(0,1)$ small, we have
$$
(\lambda-\hat{\lambda})m_0^{p(z)-1}-\chi(\delta)\geq\eta>0.
$$

Then from \eqref{eq56} and Proposition \ref{prop5}, we have
$$
u_\lambda-m_0^\delta\in D_+ \mbox{ for all }\delta\in(0,1) \mbox{ small, }
$$
a contradiction. This means that $0<\lambda^*\leq \hat{\lambda}<\infty$.
\end{proof}

According to this proposition, we have
\begin{equation}\label{eq57}
  (0,\lambda^*)\subseteq\mathcal{L}\subseteq(0,\lambda^*].
\end{equation}

We will show that for all $\lambda\in(0,\lambda^*)$, we have at least two positive smooth solutions for problem $(P_\lambda)$. To do this we need to strengthen a little the hypotheses on $f(z,\cdot)$. The new conditions on $f(z,x)$ are the following:

\smallskip
$H_3:$ $f:\Omega\times\RR\to\RR$ is a function measurable in $z\in\Omega$, for a.a. $z\in\Omega$ $f(z,\cdot)\in C^1(\RR)$, hypotheses $H_3(i)\to(v)$ are the same as the corresponding hypotheses $H_2(i)\to(v)=H_1(i)\to(v)$ and
\begin{itemize}
  \item[(vi)] for every $m>0$, there exists $\eta_m>0$ such that
$$
f(z,x)\geq\eta_m>0 \mbox{ for a.a. }z\in\Omega, \mbox{ all }x\geq m.
$$
\end{itemize}

\begin{prop}\label{prop14}
  If hypotheses $H_0$, $H_3$ hold and $\lambda\in(0,\lambda^*)$, then problem $(P_\lambda)$ admits at least two positive solutions
$$
u_0,\hat{u}\in{\rm int}\,C_+,\;u_0\not=\hat{u}.
$$
\end{prop}

\begin{proof}
  Let $\eta\in(\lambda,\lambda^*)$. We have $\eta\in\mathcal{L}$ (see \eqref{eq57}) and so we can find $u_\eta\in S_\eta\subseteq{\rm int}\,C_+$. Then according to Proposition \ref{prop9}, we can find $u_0\in S_\lambda\subseteq{\rm int}\,C_+$ such that
\begin{equation}\label{eq58}
  u_\eta-u_0\in{\rm int}\,C_+.
\end{equation}

Recall that $\overline{u}_\lambda\leq u_0$ (see Proposition \ref{prop11}). Let $\rho=\|u_0\|_\infty$ and let $\hat{\xi}_\rho>0$ be as postulated by hypothesis $H_3(v)=H_2(v)$. We can assume that $\hat{\xi}_\rho>\|\xi\|_\infty$. Then we have
\begin{eqnarray}\nonumber 
   && -\Delta_{p(z)} \overline{u}_\lambda+[\xi(z)+\hat{\xi}_\rho]\overline{u}_\lambda^{p(z)-1} \\ \nonumber
   &\leq& -\Delta_{p(z)}\overline{u}_\lambda+[|\xi(z)|+\hat{\xi}_\rho]\overline{u}_\lambda^{p(z)-1} \\ \nonumber
   &=& \lambda \overline{u}_\lambda^{q(z)-1}+\hat{\xi}_\rho\overline{u}_\lambda^{p(z)-1} \mbox{ (see Proposition \ref{prop10}) } \\ \nonumber
   &\leq& \lambda u_0^{q(z)-1}+f(z,\overline{u}_\lambda)+\hat{\xi}_\rho \overline{u}_\lambda^{p(z)-1} \mbox{ (recall that $f\geq0$) } \\ \nonumber
   &\leq& \lambda u_0^{q(z)-1}+f(z,u_0)+\hat{\xi}u_0^{p(z)-1} \\ \nonumber
   && \mbox{ (see Proposition \ref{prop11} and hypothesis $H_3(v)=H_2(v)$) } \\
   &=& -\Delta_{p(z)} u_0 +[\xi(z)+\hat{\xi}_\rho]u_0^{p(z)-1} \mbox{ (since $u_0\in S_\lambda$).} \label{eq59}
\end{eqnarray}

On account of hypothesis $H_3(vi)$ and since $\overline{u}_\lambda\in{\rm int}\,C_+$, we see that
$$
0\prec f(\cdot,\overline{u}_\lambda(\cdot)).
$$

Then \eqref{eq59} and Proposition \ref{prop4} imply that
\begin{equation}\label{eq60}
  u_0-\overline{u}_\lambda\in{\rm int}\,C_+.
\end{equation}

From \eqref{eq58} and \eqref{eq60} it follows that
\begin{equation}\label{eq61}
u_0\in{\rm int}_{C_0^1(\overline{\Omega})}[\overline{u}_\lambda,u_\eta].
\end{equation}

As before, let $\vartheta>\|\xi\|_\infty$ and consider the Carath\'eodory function $k_\lambda(z,x)$ defined by
\begin{equation}\label{eq62}
  k_\lambda(z,x)=\left\{
                   \begin{array}{ll}
                     \lambda \overline{u}_\lambda(z)^{q(z)-1}+f(z,\overline{u}_\lambda(z))+\vartheta\overline{u}_\lambda(z)^{p(z)-1}, & \hbox{ if }x<\overline{u}_\lambda(z) \\
                     \lambda x^{q(z)-1}+f(z,x)+\vartheta x^{p(z)-1}, & \hbox{ if }\overline{u}_\lambda(z)\leq x\leq u_\eta(z) \\
                     \lambda u_\eta(z)^{q(z)-1}+f(z,u_\eta(z))+\vartheta u_\eta(z)^{p(z)-1}, & \hbox{ if } u_\eta(z)<x.
                   \end{array}
                 \right.
\end{equation}

We set $K_\lambda(z,x)=\displaystyle{\int_0^x k_\lambda(z,s)ds}$ and consider the $C^1$-functional $\tau_\lambda: W^{1,p(z)}_0(\Omega)\to\RR$ defined by
$$
\tau_\lambda(u)=\int_\Omega \frac{1}{p(z)}|Du|^{p(z)}dz+\int_\Omega\frac{\vartheta+\xi(z)}{p(z)}|u|^{p(z)}dz-\int_\Omega K_\lambda(z,u)dz
$$
for all $u\in W^{1,p(z)}_0(\Omega)$.

From \eqref{eq62} and since $\vartheta>\|\xi\|_\infty$, we infer that $\tau_\lambda(\cdot)$ is coercive. Also it is sequentially weakly lower semicontinuous. So, we can find $\tilde{u}_0\in W^{1,p(z)}_0(\Omega)$ such that
\begin{eqnarray*} 
   && \tau_\lambda(\tilde{u}_0)=\min\{\tau_\lambda(u):\:u\in W^{1,p(z)}_0(\Omega)\}, \\
   &\Rightarrow& \tau'_\lambda(\tilde{u}_0)=0, \\
   &\Rightarrow& \langle\tau'_\lambda(\tilde{u}_0),h\rangle=0 \mbox{ for all }h\in W^{1,p(z)}_0(\Omega).
\end{eqnarray*}

Choosing $h=(\overline{u}_\lambda-\tilde{u}_0)^+$ and $h=(\tilde{u}_0-u_\eta)^+$ and using \eqref{eq62}, we show as before that
$$
\tilde{u}_0\in[\overline{u}_\lambda,u_\eta]\cap{\rm int}\,C_+.
$$

Therefore we may assume that $\tilde{u}_0=u_0$ or otherwise we already have a second positive smooth solution and so we are done.

Next, we consider the Carath\'eodory function
\begin{equation}\label{eq63}
  \hat{k}_\lambda(z,x)=\left\{
                         \begin{array}{ll}
                           \lambda \overline{u}_\lambda(z)^{q(z)-1}+f(z,\overline{u}_\lambda(z))+\vartheta \overline{u}_\lambda(z)^{p(z)-1}, & \hbox{ if } x\leq\overline{u}_\lambda(z) \\
                           \lambda x^{q(z)-1}+f(z,x)+\vartheta x^{p(z)-1}, & \hbox{ if }\overline{u}_\lambda(z)<x.
                         \end{array}
                       \right.
\end{equation}

We define $\hat{K}_\lambda(z,x)=\displaystyle{\int_0^x \hat{k}_\lambda(z,s)ds}$ and introduce the $C^1$-functional $\hat{\tau}_\lambda:W^{1,p(z)}_0(\Omega)\to\RR$ defined by
$$
\hat{\tau}_\lambda(u)=\int_\Omega \frac{1}{p(z)}|Du|^{p(z)}dz+\int_\Omega\frac{\vartheta+\xi(z)}{p(z)}|u|^{p(z)}dz-\int_\Omega \hat{K}_\lambda(z,u)dz
$$
for all $u\in W^{1,p(z)}_0(\Omega)$.

From \eqref{eq62} and \eqref{eq63} it is clear that
$$
\tau_\lambda|_{[\overline{u}_\lambda,u_\eta]}=\hat{\tau}_\lambda|_{[\overline{u}_\lambda,u_\eta]}.
$$

On account of \eqref{eq61}, we have that
\begin{eqnarray}\nonumber 
  && u_0 \mbox{ is a local $C_0^1(\overline{\Omega})$-minimizer of $\hat{\tau}_\lambda$}, \\
  &\Rightarrow& u_0\mbox{ is a local $W^{1,p(z)}_0(\Omega)$-minimizer of $\hat{\tau}_\lambda$. }\label{eq64}\\ \nonumber
   && \mbox{ (see Gasinski and Papageorgiou \cite[Proposition 3.3]{9Gas-Pap}). }
\end{eqnarray}

Using \eqref{eq63}, we can easily see that
\begin{equation}\label{eq65}
  K_{\hat{\tau}_\lambda}\subseteq[\overline{u}_\lambda)\cap{\rm int}\,C_+.
\end{equation}

Then from \eqref{eq63} and \eqref{eq65} we infer that we may assume that $K_{\hat{\tau}_\lambda}$ is finite or otherwise we already have an infinity of positive smooth solutions all distinct from $u_0$ and so, we are done. According to Theorem 5.7.6 of Papageorgiou, R\u adulescu and Repov\v s \cite[p. 449]{17Pap-Rad-Rep}, we can find $\rho\in(0,1)$ small such that
\begin{equation}\label{eq66}
  \hat{\tau}_\lambda(u_0)<\inf\left[\hat{\tau}_\lambda(u):\:\|u-u_0\|=\rho\right]=\hat{m}_\rho.
\end{equation}

On account of hypothesis $H_3(ii)=H_1(ii)$, for $u\in {\rm int}\,C_+$, we have
\begin{equation}\label{eq67}
  \hat{\tau}_\lambda(tu)\to-\infty \mbox{ as }t\to+\infty.
\end{equation}

Finally, from \eqref{eq63} it follows that
\begin{eqnarray}\nonumber 
  && \hat{\varphi}_\lambda|_{[\overline{u}_\lambda)}=\hat{\tau}_\lambda|_{[\overline{u}_\lambda)}+\hat{\eta} \mbox{ with }\hat{\eta}\in\RR, \\
  &\Rightarrow& \hat{\tau}_\lambda(\cdot) \mbox{ satisfies the $C$-condition (see Proposition \ref{prop6}).} \label{eq68}
\end{eqnarray}

Then \eqref{eq66}, \eqref{eq67}, \eqref{eq68} permit the use of the mountain pass theorem. So, we can find $\hat{u}\in W^{1,p(z)}_0(\Omega)$ such that
\begin{equation}\label{eq69}
  \hat{u}\in K_{\hat{\tau}_\lambda}\subseteq[\overline{u}_\lambda)\cap{\rm int}\,C_+ \mbox{ and }\hat{m}_\rho\leq\hat{\tau}_\lambda(\hat{u}).
\end{equation}

From \eqref{eq69} and \eqref{eq63} we see that $\hat{u}\in S_\lambda\subseteq {\rm int}\,C_+$, while from \eqref{eq69} and \eqref{eq66} we have that $\hat{u}\not=u_0$.
\end{proof}

Finally, we show that the critical parameter value $\lambda^*$ is admissible, that is, $\lambda^*\in\mathcal{L}$.

\begin{prop}\label{prop15}
  If hypotheses $H_0$, $H_1$ hold, then $\lambda^*\in\mathcal{L}$.
\end{prop}

\begin{proof}
  Let $\{\lambda_n\}_{n\geq1}\subseteq\mathcal{L}$ such that $\lambda_n\uparrow \lambda^*$ as $n\to\infty$. From the proof of Proposition \ref{prop8}, we know that we can find $u_n\in S_{\lambda_n}\subseteq{\rm int}\,C_+$ such that
$$
\hat{\varphi}_{\lambda_n}(u_n)<0 \mbox{ for all }n\in\NN.
$$

Also, we have
$$
\hat{\varphi}_{\lambda_n}'(u_n)=0, \mbox{ for all }n\in\NN.
$$

Then as in the proof of Proposition \ref{prop6}, we show that
$$
\{u_n\}_{n\geq1}\subseteq W^{1,p(z)}_0(\Omega) \mbox{ is bounded. }
$$

We may assume that
\begin{equation}\label{eq70}
  u_n\overset{w}{\to}u_* \mbox{ in }W^{1,p(z)}_0(\Omega) \mbox{ and }u_n\to u_* \mbox{ in } L^{r(z)}(\Omega) \mbox{ as }n\to\infty.
\end{equation}

We have
$$
\langle A(u_n),h\rangle+\int_\Omega \xi(z)u_n^{p(z)-1}hdz=\lambda_n\int_\Omega u_n^{q(z)-1}hdz+\int_\Omega f(z,u_n)hdz
$$
for all $h\in W^{1,p(z)}_0(\Omega)$, all $n\in\NN$.

Choosing $h=u_n-u_*$, passing to the limit as $n\to\infty$ and using \eqref{eq70} and Proposition \ref{prop3}, we obtain
$$
u_n\to u_* \mbox{ in } W^{1,p(z)}_0(\Omega).
$$

So, in the limit as $n\to\infty$, we have
$$
\langle A(u_*),h\rangle+\int_\Omega \xi(z)u_*^{p(z)-1}hdz=\lambda^*\int_\Omega u_*^{q(z)-1}hdz+\int_\Omega f(z,u_*)hdz
$$
for all $h\in W^{1,p(z)}_0(\Omega)$.

We have
\begin{eqnarray*} 
   && \overline{u}_{\lambda_1}\leq u_n \mbox{ for all }n\in\NN \\
   && \mbox{ (see the Remark after Proposition \ref{prop11}), }\\
   &\Rightarrow& \overline{u}_{\lambda_1}\leq u_*, \\
   &\Rightarrow& u_*\in S_{\lambda^*}\subseteq{\rm int}\,C_+ \mbox{ and so }\lambda^*\in\mathcal{L}.
\end{eqnarray*}
The proof is now complete.
\end{proof}

According to this proposition, we have
$$
\mathcal{L}=(0,\lambda^*].
$$

Summarizing, we can state the following bifurcation-type result describing in a precise way the set of the positive solutions of problem $(P_\lambda)$ as the parameter $\lambda>0$ varies.

\begin{thm}\label{th1}
  If hypotheses $H_0$, $H_3$ hold, then there exists $\lambda^*>0$ such that
\begin{itemize}
  \item[(a)] for all $\lambda\in(0,\lambda^*)$, problem $(P_\lambda)$ has at least two positive solutions
$$
u_0,\hat{u}\in {\rm int}\,C_+,\;u_0\not=\hat{u};
$$
  \item[(b)] for $\lambda=\lambda^*$, problem $(P_\lambda)$ has at least one positive solution
$$
u_*\in{\rm int}\,C_+;
$$
  \item[(c)] for $\lambda>\lambda^*$, problem $(P_\lambda)$ has no positive solutions;
  \item[(d)] for every $\lambda\in\mathcal{L}=(0,\lambda^*]$, problem $(P_\lambda)$ has a smallest positive solution $u_\lambda^*\in{\rm int}\,C_+$ and the map $\lambda\mapsto u_\lambda^*$ from $\mathcal{L}=(0,\lambda^*]$ into $C_+\setminus\{0\}$ is increasing, that is,
$$
0<\mu\leq\lambda\in\mathcal{L}\Rightarrow u_\lambda^*-u_\mu^*\in C_+\setminus\{0\}.
$$
\end{itemize}
\end{thm}

\medskip
{\bf Acknowledgments.} This research was supported by the Slovenian Research Agency grants
P1-0292, J1-8131, N1-0064, N1-0083, and N1-0114.


\begin{thebibliography}{99}   {\small
\bibitem{1Amb-Bre-Cer} A. Ambrosetti, H. Brezis, G. Cerami,
Combined effects of concave convex nonlinearities in some elliptic problems, {\it J. Functional Anal.} {\bf 122} (1994), 519-543.

\bibitem{2Amb-Rab} A. Ambrosetti, P. Rabinowitz,
Dual variational methods in critical point theory and applications, {\it J. Functional Anal.} {\bf 14} (1973), 349-381.

\bibitem{3Arc-Rui} D. Arcoya, D. Ruiz, The Ambrosetti-Prodi problem for the $p$-Laplace operator, {\it Comm. Partial Diff. Equations} {\bf 31} (2006), 849-865.

\bibitem{BRR1} A. Bahrouni, V.D. R\u adulescu, D.D. Repov\v{s}, A weighted anisotropic variant of the Caffarelli-Kohn-Nirenberg inequality and applications, {\it Nonlinearity} {\bf 31} (2018), no. 4, 1516-1534.

\bibitem{BRR2} A. Bahrouni, V.D. R\u adulescu, D.D. Repov\v{s}, Double phase transonic flow problems with variable growth: nonlinear patterns and stationary waves, {\it Nonlinearity} {\bf 32} (2019), no. 7, 2481-2495.
    
\bibitem{Cherfils} L. Cherfils, A. Miranville, S. Peng, Higher-order anisotropic models in phase separation, {\it Adv. Nonlinear Anal.} {\bf 8} (2019), no. 1, 278-302.

\bibitem{4Die-Har-Has-Ruz} L. Diening, P. Harjulehto, P. H\"asto, M. Ruzi\v cka,
{\it Lebesgue and Sobolev Spaces with Variable Exponents}, Lecture Notes in Mathematics, Vol. 2017, Springer, Heidelberg, 2011.

\bibitem{5Fan} X. Fan, Global $C^{1,\alpha}$ regularity for variable exponent elliptic equations in divergence form, {\it J. Differential Equations} {\bf 235} (2007), 397-417.

\bibitem{6Fan-Zha} X. Fan, D. Zhao, A class of De Giorgi type and H\"older continuity, {\it Nonlinear Anal.} {\bf 36} (1999), 295-318.

\bibitem{7Fil-Pap} M. Filippakis, N.S. Papageorgiou, Multiple constant sign and nodal solutions for nonlinear elliptic equations with the $p$-Laplacian, {\it J. Differential Equations} {\bf 245} (2008), 1883-1929.

\bibitem{8Gar-Man-Per} J. Garcia Azorero, J. Manfredi, I. Peral Alonso, Sobolev versus H\"older local minimizers and global multiplicity for some quasilinear elliptic equations, {\it Comm. Contemp. Math.} {\bf 2} (2000), 385-404.

\bibitem{9Gas-Pap} L. Gasinski, N.S. Papageorgiou, Anisotropic nonlinear Neumann problems, {\it Calc. Var.} {\bf 42} (2011), 323-354.

\bibitem{10Gue-Ver} M. Guedda, L. V\'eron, Quasilinear elliptic equations involving critical Sobolev exponents, {\it Nonlinear Anal.} {\bf 13} (1989), 879-902.

\bibitem{11Guo-Zha} Z. Guo, Z. Zhang, $W^{1,p}$ versus $C^1$ local minimizers and multiplicity results for quasilinear elliptic equations, {\it J. Math. Anal. Appl.} {\bf 286} (2003), 32-50.

\bibitem{12Hu-Pap} S. Hu, N.S. Papageorgiou,
{\it Handbook of Multivalued Analysis. Volume I: Theory}, Kluwer Academic Publishers, Dordrecht, The Netherlands, 1997.

\bibitem{13Li-Yan} G. Li, C. Yang, The existence of a nontrivial solution to a nonlinear elliptic boundary value problem of $p$-Laplacian type without Ambrosetti-Rabinowitz condition, {\it Nonlinear Anal.} {\bf 72} (2010), 4602-4613.

\bibitem{14Mar-Pap} S. Marano, N.S. Papageorgiou, Positive solutions to a Dirichlet problem with $p$-Laplacian and concave-convex nonlinearity depending on a parameter, {\it Comm. Pure Appl. Anal.} {\bf 12} (2013), 815-829.

\bibitem{15Pap-Rad} N.S. Papageorgiou, V.D. R\u adulescu, Bifurcation of positive solutions for nonlinear nonhomogeneous Robin and Neumann problems with competing nonlinearities, {\it Discr. Cont. Dyn. Syst.} {\bf 35} (2016), 5008-5036.

\bibitem{16Pap-Rad-Rep} N.S. Papageorgiou, V.D. R\u adulescu, D.D. Repov\v s, Positive solutions for perturbations of the Robin eigenvalue problem plus an indefinite potential, {\it Discr. Cont. Dyn. Syst.} {\bf 37} (2017), 2589-2618.
    
 \bibitem{PRR0} N.S. Papageorgiou, V.D. R\u adulescu, D.D. Repov\v s,  $(p,2)$-equations asymmetric at both zero and infinity, {\it Adv. Nonlinear Anal.} {\bf 7} (2018), no. 3, 327-351.

\bibitem{PRR1} N.S. Papageorgiou, V.D. R\u adulescu, D.D. Repov\v s, Double-phase problems with reaction of arbitrary growth, {\it Z. Angew. Math. Phys.} {\bf 69} (2018), no. 4, Paper No. 108, 21 pp.

\bibitem{PRR2} N.S. Papageorgiou, V.D. R\u adulescu, D.D. Repov\v s, Double-phase problems and a discontinuity property of the spectrum, {\it Proc. Amer. Math. Soc.} {\bf 147} (2019), no. 7, 2899-2910.
    
 \bibitem{PRR2bis} N.S. Papageorgiou, V.D. R\u adulescu, D.D. Repov\v s,   Positive solutions for nonlinear parametric singular Dirichlet problems, {\it Bull. Math. Sci.} {\bf 9} (2019), no. 3, 1950011, 21 pp.

\bibitem{17Pap-Rad-Rep} N.S. Papageorgiou, V.D. R\u adulescu, D.D. Repov\v s, {\it Nonlinear Analysis--Theory and Methods}, Springer Monographs in Mathematics, Springer, Cham, 2019.
    
\bibitem{PRR3} N.S. Papageorgiou, V.D. R\u adulescu, D.D. Repov\v s,     Ground state and nodal solutions for a class of double phase problems, {\it Z. Angew. Math. Phys.} {\bf 71} (2020), no. 1, Paper No. 15, 15 pp.

\bibitem{18Rad-Rep} V.D. R\u adulescu, D.D. Repov\v s, {\it
Partial Differential Equations with Variable Exponents: Variational Methods and Qualitative Analysis}, CRC Press, Boca Raton, FL, 2015.

\bibitem{19Tak-Gia} P. Taka\v c, J. Giacomoni, A $p(x)$-Laplacian extension of the Diaz-Saa inequality and some applications, {\it Proc. Royal Soc. Edinburgh}, DOI:  https://doi.org/10.1017/prm.2018.91.

\bibitem{20Vaz} J.-L. V\'azquez,
A strong maximum principle for some quasilinear elliptic equations, {\it Appl. Math. Optim.} {\bf 12} (1984), 191-202.

\bibitem{21Zha} Q. Zhang, A strong maximum principle for differential equations with nonstandard $p(x)$-growth conditions, {\it J. Math. Anal. Appl.} {\bf 312} (2005), 125-143.

\bibitem{ZRJMPA} Q. Zhang, V.D. R\u adulescu, Double phase anisotropic variational problems and combined effects of reaction and absorption terms, {\it J. Math. Pures Appl.} {\bf 118} (2018), 159-203.
}

\end{thebibliography}
\end{document}